\documentclass[12pt]{article}

\usepackage[a4paper,margin=28mm]{geometry}
\usepackage{amsmath,amssymb,amsthm,mathtools,mathrsfs,cite}
\usepackage{bbm}
\usepackage{enumitem}
\usepackage{microtype}
\usepackage[hidelinks]{hyperref}
\hypersetup{
  pdftitle={Dirichlet-Constrained Random Operators and Quenched Reflection Quantum Walks: Tree Spectra and Poisson Criticality at High Degree},
  pdfauthor={Nino Nagata, Ryoichi Suzuki},
  pdfkeywords={random operators on trees; Dirichlet random environment; Poisson-weighted tree; reflection quantum walk; semicircle law; high-degree transition}
}

\newtheorem{theorem}{Theorem}[section]
\newtheorem{proposition}[theorem]{Proposition}
\newtheorem{corollary}[theorem]{Corollary}
\newtheorem{lemma}[theorem]{Lemma}
\theoremstyle{definition}
\newtheorem{definition}[theorem]{Definition}
\theoremstyle{remark}
\newtheorem{remark}[theorem]{Remark}

\newcommand{\T}{\mathbb{T}}
\newcommand{\C}{\mathbb{C}}
\newcommand{\R}{\mathbb{R}}
\newcommand{\E}{\mathbb{E}}
\newcommand{\Pp}{\mathbb{P}}
\newcommand{\1}{\mathbbm{1}}
\newcommand{\Tr}{\operatorname{Tr}}
\newcommand{\supp}{\operatorname{supp}}
\newcommand{\rank}{\operatorname{rank}}
\newcommand{\Cat}{\operatorname{Cat}}
\newcommand{\dd}{\,\mathrm{d}}
\newcommand{\Imag}{\operatorname{Im}}
\newcommand{\N}{\mathbb{N}}
\newcommand{\edg}{\mathfrak{e}}

\title{Dirichlet-Constrained Random Operators and Quenched Reflection Quantum Walks:\\
Tree Spectra and Poisson Criticality at High Degree}
\author{%
Nino Nagata$^{1}$ and Ryoichi Suzuki$^{2,*}$\\[0.55em]
\small $^{1}$Department of Mathematical Sciences, College of Science and Engineering,\\
\small Ritsumeikan University, Kusatsu, Shiga 525-8577, Japan\\
\small $^{2}$School of Management, Tokyo University of Science, Tokyo 102-0071, Japan\\[0.35em]
\small $^{*}$Corresponding author:
\href{mailto:rsuzukimath@gmail.com}{\texttt{rsuzukimath@gmail.com}}\\[0.35em]
\footnotesize ORCID: Nino Nagata 0009-0009-5064-0499;\\
\footnotesize Ryoichi Suzuki 0000-0001-9979-1882}
\date{}

\begin{document}
\maketitle

\begin{abstract}
We study spectral limits of random operators whose incident edge weights
are coupled by local normalization, and the reflection quantum walks
associated with them. At each vertex of a \(d\)-regular graph, an
independent Dirichlet probability vector distributes unit mass among
outgoing directions. At fixed degree, the empirical spectral measures
of the associated self-adjoint operators converge almost surely on
locally tree-like graph sequences. We give a closed-walk formula for
every moment of the limiting tree law.

Our main result identifies two distinct high-degree scales. After
rescaling the deterministic expected root spectral law by \(\sqrt d\),
a weak limit exists if and only if \(d\alpha_d^2\) converges in
\([0,\infty]\), where \(\alpha_d\) is the Dirichlet parameter. The limits range from a
point mass at zero through non-semicircular critical laws to the
semicircle law. The critical laws admit a Poisson-tree representation
with an explicit Bessel intensity. The extended family, including both
endpoint laws, provides a first-Wasserstein approximation uniform over
positive Dirichlet parameters.

In the regime \(d\alpha_d^2\to0\), the rescaled law converges weakly
to a point mass although its second moment remains one. Weighting
spectral mass by its contribution to this moment and rescaling again
reveals a different family of limits, classified by \(d\alpha_d\).
These limits have explicit Beta-product descriptions. A quantitative
comparison transfers the variance-scale classification to the
centered and rescaled inherited eigenphases of the quantum walk.
\end{abstract}



\section{Introduction}\label{sec:introduction}

We study how a local normalization constraint affects the spectrum of a
random operator on a graph. At each vertex, a random probability vector
distributes unit mass among the outgoing directions. An edge weight uses
one mass from each endpoint. Edges meeting at a vertex are therefore
dependent, even when the vertex vectors are sampled independently. The
same construction defines a reflection quantum walk in an environment that
is sampled once and then kept fixed.

Let \(G\) be a \(d\)-regular graph, with \(d\geq3\). At each vertex \(v\),
choose independently
\[
 X_v=(X_{v,u})_{u\sim v}
 \sim \operatorname{Dirichlet}(\alpha,\ldots,\alpha),
 \qquad \alpha>0.
\]
The coordinate \(X_{v,u}\) is the mass assigned at \(v\) to its neighbor
\(u\). For fixed \(d\), as \(\alpha\downarrow0\), a vertex places nearly
all its mass in one randomly chosen direction. As \(\alpha\to\infty\),
the masses approach the uniform allocation. In the positive-weight case, the operator has
zero diagonal and edge coefficients
\[
 T_{v,u}=\sqrt{X_{v,u}X_{u,v}},\qquad u\sim v.
\]
It is sparse and self-adjoint, unlike the dense Dirichlet Markov matrices
studied in \cite{Chafai2010}. The full model also allows complex phases.
On a tree these phases can be removed by a diagonal unitary transformation,
so they do not change the tree spectral law.

Choose amplitudes \(q_{v,u}\) with \(|q_{v,u}|^2=X_{v,u}\). Then
\(q_v=(q_{v,u})_{u\sim v}\) is a unit vector and defines a local reflection
coin. Applying these reflections and then reversing the directed edges
gives the quantum walk.
The environment remains fixed throughout the evolution; this is the
\emph{quenched} setting.

The associated self-adjoint operator \(T\) is called the \emph{discriminant}. Its spectral values give the
\emph{inherited} part of the unitary spectrum through
\(x\mapsto x\pm i\sqrt{1-x^2}\) \cite{SegawaSuzuki2019}.
When the amplitudes are chosen positive, \(X_{v,u}\) is also a classical
transition probability in a
Dirichlet random environment \cite{SabotTournier2017}, and the quantum
walk is its reflection/Szegedy-type quantization
\cite{PortugalSegawa2017,Szegedy2004}.

The central question is how the local dependence changes the spectral
behavior as the degree grows. Our main result identifies two distinct
scales for the deterministic expected root spectral laws. The first
classifies their variance-normalized weak limits. In the regime where
this limit collapses to a point mass at zero, the second classifies
the spectral mass carrying the still nonzero second moment. Thus
different Dirichlet parameter sequences can have the same ordinary
spectral limit but different rescaled, second-moment-weighted limits.

\paragraph{Main results.}
We take the large-graph limit first, with \(d\) and \(\alpha\) fixed, and
then study the resulting tree laws as \(d\to\infty\). The high-degree
statements thus concern deterministic measures obtained by averaging the
root spectral measure over the tree environment.

\smallskip
\noindent\textit{Fixed degree: the tree law and its support.}
Consider a graph sequence in which the fraction of vertices with a tree
neighborhood of any fixed radius tends to one. The empirical spectral
measures of the discriminants converge almost surely to a deterministic
measure \(\mu_{d,\alpha}\). This is the expected spectral measure at the
root of the marked regular tree, which we call the \emph{tree law}
(Theorem~\ref{thm:quenched-limit}). We compute every moment by summing the
weights of closed walks. The formula retains the mixed Dirichlet moments
at each vertex, so it accounts for the dependence between incident edges
(Theorem~\ref{thm:moment-formula}).

For every \(\alpha>0\), the tree operator almost surely has essential
spectrum \([-1,1]\), and the tree law has support \([-1,1]\). Long paths
with internal edge weights close to one half support approximate
eigenvectors throughout the interval. Small side-branch weights control
lateral leakage, while the endpoint amplitudes of the path eigenvectors
control the two boundary contributions
(Theorems~\ref{thm:full-spectrum}--\ref{thm:full-support}).
Despite this common support, the endpoint laws have smaller support:
as \(\alpha\downarrow0\) the limit is a three-point law on
\(\{-1,0,1\}\), and as \(\alpha\to\infty\) it is the normalized
Kesten--McKay law of the unweighted regular tree
(Theorem~\ref{thm:disorder-limits}).

\smallskip
\noindent\textit{High degree: the spectral law and the mass carrying its second moment.}
The second moment of \(\mu_{d,\alpha}\) is \(1/d\), so multiplication of
the spectral variable by \(\sqrt d\) makes this moment equal to one.
For \(\alpha=\alpha_d\), the rescaled tree laws converge weakly if and
only if \(d\alpha_d^2\) converges in \([0,\infty]\). When this parameter
tends to infinity, the limit is the semicircle law. A finite positive
limit gives a non-semicircular critical law, and a zero limit gives
\(\delta_0\) (Corollary~\ref{cor:sqrt-d-phase-diagram}).

The critical laws and their two endpoint laws form a continuous family
in the first Wasserstein distance. At degree \(d\), the member indexed
by \(d\alpha^2\) approximates the rescaled tree law, with an error tending
to zero uniformly over \(\alpha>0\)
(Proposition~\ref{prop:uniform-phase-approximation}).

For a finite positive limit \(d\alpha_d^2\to\tau\), we identify the critical law as the average
root spectral measure of an essentially self-adjoint operator on a
Poisson-weighted tree. The intensity of the squared edge weights is given
explicitly by a Bessel-function density
(Theorem~\ref{thm:poisson-critical}).

The case \(d\alpha_d^2\to0\) reveals why a second parameter is needed.
Most spectral mass approaches zero, yet the second moment remains one.
A vanishing amount of mass at large rescaled spectral values carries this
moment.

To describe this mass, we weight the tree law by the square of the spectral
variable, normalize the result to a probability measure, and change the
observation scale. These \emph{energy-biased laws} converge if and only if
\(d\alpha_d\) converges in \([0,\infty]\), under the condition
\(d\alpha_d^2\to0\). Their limits are explicit laws built from products
of Beta variables, together with their endpoint limits
(Theorem~\ref{thm:subcritical-energy} and
Corollary~\ref{cor:energy-classification}). For example, all choices
\(\alpha_d=c/d\) with fixed \(c>0\) have the same variance-normalized
weak limit \(\delta_0\), but distinct energy-biased limits for distinct
\(c\).

\smallskip
\noindent\textit{Quantum-walk eigenphases and tree recursions.}
At high degree, the full unitary spectral law is dominated by
structural atoms at \(\pm1\). The high-degree transition studied here
is captured by extracting and normalizing the inherited component.
We center and rescale these eigenphases and compare their law with
the rescaled tree law. Their first Wasserstein distance is \(O(d^{-1/2})\), uniformly in
\(\alpha\), and improves to \(O(d^{-1})\) when the rescaled fourth
moments remain bounded. Hence the same high-degree spectral
classification and uniform approximation hold for these eigenphases
(Corollary~\ref{cor:eigenphase-diagram}).

Removing an edge of the tree also gives a recursion for branch Green
functions. The branch Green function remains dependent on the Dirichlet
mass assigned to the parent direction. Their product satisfies a closed
recursion. We prove uniqueness of its law as a random analytic function
with positive imaginary part and the resolvent bound. We also prove the
corresponding uniqueness result for the critical Poisson model
(Theorems~\ref{thm:cavity-fixed-point} and \ref{thm:poisson-rde}).

\paragraph{Related work.}
The spectral relation between reflection walks and self-adjoint
discriminants comes from abstract quantum-walk theory
\cite{SegawaSuzuki2019,KonnoSatoSegawaShikano2023}. For sparse random
operators, local weak convergence and resolvent methods yield spectral
limits \cite{BordenaveLelarge2010,AnantharamanEtAl2021}. Closed-walk
expansions provide another established approach to sparse and weighted
random matrices \cite{KhorunzhyShcherbinaVengerovsky2004,Zakharevich2006}.
We use these methods while keeping the mixed moments imposed by the
Dirichlet constraint.

Poisson-weighted trees and recursive distributional equations are also
established tools for random-operator limits
\cite{AldousSteele2004,AldousBandyopadhyay2005,
BordenaveCaputoChafai2011}. In particular, Jung \cite{Jung2018}
represents spectral limits of certain random matrix ensembles using
Poisson-weighted trees, with matrix entries independent up to Hermitian
symmetry. Here, the incident coefficients share
Dirichlet coordinates, and their mixed moments enter the closed-walk
expansion. The resulting critical law has a specific Bessel intensity,
while a second scaling recovers the spectral energy that is invisible
in the collapsed weak limit.

The Kesten--McKay-to-semicircle transition gives the comparison without
Dirichlet disorder
\cite{Kesten1959,McKay1981,DumitriuPal2012,GongLiLiu2025}.
Other work on tree operators concerns spectral stability, resonant
delocalization, and trees of finite cone type
\cite{AizenmanSimsWarzel2006,AizenmanWarzel2013,KellerLenzWarzel2013}.
Our support theorem identifies where the spectrum lies, without
distinguishing pure point, absolutely continuous, and singular continuous
parts. The Dirichlet Markov ensemble of Chafa\"i
\cite{Chafai2010} and classical random walks in Dirichlet environment
\cite{SabotTournier2017} provide further comparisons. Here we study the
sparse self-adjoint discriminant and its unitary reflection walk.

The way randomness enters a quantum walk is important. Hamza and Joye
study a different static-disorder model on homogeneous trees
\cite{HamzaJoye2014}. Quillen resamples a full Haar coin at every time
step and studies the averaged channel \cite{Quillen2026}; our environment
is held fixed. In a 2026 preprint (arXiv:2607.02356), Maltsev and Osman
study mesoscopic linear statistics for quantum graphs, including a
Haar-scattering ensemble. Our local coins, by contrast, are reflections
generated by rank-one projections. Higuchi and Segawa \cite{HiguchiSegawa2018} use Dirichlet boundary
conditions and finite-depth approximations on infinite trees. Their use
of ``Dirichlet'' refers to a boundary condition, not to the probability
distribution at a vertex.

Section~\ref{sec:model-reduction} defines the model and its spectral
mapping. Section~\ref{sec:infinite-volume} proves the fixed-degree results.
Section~\ref{sec:high-degree} establishes the two high-degree limits,
the Poisson representation, and the uniform approximation.
Section~\ref{sec:conclusion} discusses further questions, and
Appendix~\ref{sec:cavity} develops the finite-degree recursion.

\paragraph{Notation.}
We use \(\N=\{1,2,\ldots\}\), \(\T=\{z\in\C:|z|=1\}\), and
\(\C_+=\{z\in\C:\Imag z>0\}\). We write \(\mathcal B(S)\) for the Borel
sets of a topological space \(S\), and \(f_\#\mu\) for the pushforward of
\(\mu\) under a measurable map \(f\). The symbols \(\mu_{d,\alpha}\) and
\(\nu_{d,\alpha}\) denote the limiting discriminant and unitary laws,
respectively.  For a canonical closed-walk class \(\gamma\),
\(H_\gamma\) is its traced rooted tree and
\(\edg(\gamma)=|E(H_\gamma)|\).  We reserve \((a)_r\) for the rising
factorial and write \(x^{\underline r}\) for the falling factorial.  The
high-degree parameters \(\varepsilon_{d,\alpha}\), \(\tau_d\),
\(\theta_d\), and \(s_d\) are defined when first used in
Section~\ref{sec:high-degree}. For probability measures on \(\R\) with finite
first moment, \(W_1(\mu,\nu)=\inf\E|X-Y|\), where the infimum is over all
couplings with laws \(\mu\) and \(\nu\). On other metric spaces, we replace
\(|x-y|\) by the specified metric. We use the standard fact that weak
convergence, together with uniform integrability of first moments, implies
\(W_1\) convergence \cite[Chapter~6]{Villani2009}.

\section{Random reflection model and discriminant reduction}
\label{sec:model-reduction}

\subsection{The random reflection model}

Fix an integer \(d\geq 3\).  Let \(G=(V,E)\) be a finite simple
\(d\)-regular graph, and let
\[
 A^{\rightarrow}
 =
 \{(v,u):\{v,u\}\in E\}
\]
be its set of directed edges.  The walk space is
\[
 \mathcal{H}_{G}=\ell^2(A^{\rightarrow}).
\]
Fix an ordering of the \(d\) outgoing arcs at every vertex.  The
marked stars
\[
 \omega_v=(X_v,\xi_v),
 \qquad
 X_v=(X_{v,u})_{u\sim v},
 \qquad
 \xi_v=(\xi_{v,u})_{u\sim v},
\]
are independent and identically distributed over the vertices.  The marginal law
of the squared amplitudes is
\[
 X_v\sim\operatorname{Dirichlet}(\alpha,\ldots,\alpha),
 \qquad \alpha>0,
\]
and \(\xi_v\in\T^d\) may have an arbitrary conditional law given \(X_v\).
The phase coordinates may be dependent among themselves and on \(X_v\);
deterministic phase choices are included as a special case.  Set
\[
 q_{v,u}=\sqrt{X_{v,u}}\,\xi_{v,u},
 \qquad
 q_v=(q_{v,u})_{u\sim v}.
\]
Then \(\|q_v\|=1\).  The local reflection coin is
\[
 C_v=2q_vq_v^*-I_d,
\]
and the global coin is \(C=\bigoplus_{v\in V}C_v\).  The flip--flop
shift is
\[
 (S\psi)(v,u)=\psi(u,v).
\]
Thus \(S=S^*=S^{-1}\), and the one-step quantum walk is
\[
 U=SC.
\]

The canonical amplitude-disorder model is obtained by taking
\(\xi_{v,u}=1\).  The real and complex Haar-reflection models arise,
respectively, from \(\alpha=1/2\) with independent Rademacher signs and from
\(\alpha=1\) with independent uniform phases. All infinite-tree results and
locally tree-like limiting results below are independent of the admissible
conditional phase law: phases cancel on the tree neighborhoods that determine
the limiting moments.

The phrase \emph{weak amplitude disorder} refers to the limit
\(\alpha\to\infty\).  In the phase-aligned model \(\xi_{v,u}=1\), the law of
the local coin converges weakly to the point mass at the Grover coin.  For
general admissible phase laws, the tree spectral law is unchanged because the
gauge transformation below removes the phases.

\begin{remark}[Dirichlet random environment and reflection quantization]
In the phase-aligned model, $p(v,u):=X_{v,u}$ is a quenched Markov kernel
whose transition vectors are independent Dirichlet variables. With the boundary
operator $K$ defined below, $U=S(2K^*K-I)$ is the associated flip--flop reflection
quantization, and its
discriminant is
\[
 T_{v,u}=\sqrt{p(v,u)p(u,v)}.
\]
This connects Dirichlet random environments to the Szegedy reflection
framework.  The model with phases retains the same tree law, but products of phases
around cycles may affect the operator on finite graphs.
\end{remark}

\subsection{Discriminant and spectral mapping}

The next identities reduce the spectral analysis of the unitary walk to a
self-adjoint operator on vertices. The characteristic-polynomial formula
also keeps track of the contributions at the endpoints \(\pm1\).

Define \(K:\mathcal{H}_G\to\ell^2(V)\) by
\[
 (K\psi)(v)
 =
 \sum_{u\sim v}\overline{q_{v,u}}\psi(v,u).
\]
Then
\[
 (K^*f)(v,u)=q_{v,u}f(v),
 \qquad
 KK^*=I_{\ell^2(V)},
\]
and
\[
 C=2K^*K-I_{\mathcal{H}_G}.
\]
The discriminant is
\[
 T=KSK^*.
\]

\begin{proposition}[Self-adjoint reduction]\label{prop:self-adjoint-discriminant}
The operator \(T\) is a self-adjoint contraction on \(\ell^2(V)\), and
\[
 T_{v,u}
 =
 \begin{cases}
 \overline{q_{v,u}}q_{u,v},&u\sim v,\\
 0,&u\not\sim v.
 \end{cases}
\]
In particular,
\[
 |T_{v,u}|=\sqrt{X_{v,u}X_{u,v}}.
\]
\end{proposition}

\begin{proof}
Since \(S=S^*\), one has \(T^*=KS^*K^*=T\).  Moreover,
\(\|T\|\leq\|K\|\|S\|\|K^*\|=1\).  For \(u\sim v\),
\[
 \langle\delta_v,T\delta_u\rangle
 =\overline{q_{v,u}}q_{u,v},
\]
and the entry is zero otherwise.  In particular,
\(T_{u,v}=\overline{T_{v,u}}\).
\end{proof}

\begin{lemma}[Gauge reduction on a tree]\label{lem:tree-gauge}
Let \(G\) be a finite or locally finite infinite tree carrying the
normalized vertex amplitudes described above, and fix a root \(o\).
There exists a diagonal unitary
operator \(D\) on \(\ell^2(V)\), with \(D\delta_o=\delta_o\), such that
\[
 (D^*TD)_{v,u}=|T_{v,u}|=\sqrt{X_{v,u}X_{u,v}}
 \qquad (u\sim v).
\]
Consequently, the spectrum of \(T\) and its spectral measure at the root
depend solely on the squared amplitudes \(X\).
\end{lemma}

\begin{proof}
On each edge with \(a_{v,u}:=|T_{v,u}|>0\), write
\(T_{v,u}=a_{v,u}e^{\mathrm{i}\vartheta_{v,u}}\), where
\(\vartheta_{u,v}=-\vartheta_{v,u}\). Zero-weight edges impose no phase
constraint. Set the diagonal phase at the root to one. Proceeding away from
the root, choose each child's phase so that the conjugated coefficient on the
edge to its parent is nonnegative. The absence of cycles makes these choices
consistent. The resulting diagonal unitary fixes the root vector and therefore
preserves its spectral measure.
\end{proof}

Let \(n=|V|\) and \(m=|E|=dn/2\).

\begin{proposition}[Characteristic polynomial and trace identity]\label{prop:trace-identity}
For every \(z\in\C\),
\[
 \det(zI_{2m}-U)
 =
 (z^2-1)^{m-n}
 \det\!\left((z^2+1)I_n-2zT\right).
\]
If \(\mathsf{T}_k\) denotes the Chebyshev polynomial of the first kind, then
\[
 \Tr U^k
 =
 (m-n)\bigl(1+(-1)^k\bigr)
 +
 2\Tr\mathsf{T}_k(T),
 \qquad k\geq 1.
\]
\end{proposition}

\begin{proof}
For \(z^2\neq1\),
\[
 zI-U=zI+S-2SK^*K.
\]
Since
\[
 (zI+S)^{-1}=\frac{zI-S}{z^2-1},
\]
and the flip--flop shift is a direct sum of \(m\) two-dimensional swap blocks,
we have \(\det(zI+S)=(z^2-1)^m\). The matrix determinant lemma then gives
\begin{align*}
 \det(zI-U)
 &=
 \det(zI+S)
 \det\!\left(
 I_n-2K(zI+S)^{-1}SK^*
 \right)\\
 &=
 (z^2-1)^m
 \det\!\left(
 I_n-\frac{2(zT-I_n)}{z^2-1}
 \right)\\
 &=
 (z^2-1)^{m-n}
 \det\!\left((z^2+1)I_n-2zT\right).
\end{align*}
Both sides are polynomials, so the identity extends to \(z=\pm1\).

If \(\lambda=\cos\theta\in[-1,1]\) is an eigenvalue of \(T\), the associated
quadratic factor has roots \(e^{\pm i\theta}\).  Hence its contribution to
\(\Tr U^k\) is
\[
 e^{ik\theta}+e^{-ik\theta}
 =
 2\mathsf{T}_k(\lambda).
\]
The factor \((z^2-1)^{m-n}\) contributes the remaining term.
\end{proof}

\begin{corollary}[Exact relation between empirical laws]\label{cor:empirical-law-map}
Let
\[
 \mu_T=\frac1n\sum_{j=1}^n\delta_{\lambda_j(T)}
\]
and let \(\nu_U\) be the empirical spectral measure of \(U\), normalized by
\(2m=dn\).  With
\[
 \phi_\pm(x)=x\pm i\sqrt{1-x^2},
\]
one has
\[
 \nu_U
 =
 \frac{d-2}{2d}(\delta_1+\delta_{-1})
 +
 \frac1d(\phi_+)_\#\mu_T
 +
 \frac1d(\phi_-)_\#\mu_T.
\]
When \(T\) has an eigenvalue at \(\pm1\), the branches \(\phi_+\) and
\(\phi_-\) coincide. The corresponding quadratic factor is \((z\mp1)^2\),
so each endpoint eigenvalue contributes twice to the formula.
\end{corollary}

\begin{definition}[Normalized inherited spectral law]
The normalized inherited spectral law is defined by the two inverse
Joukowsky branches:
\[
 \nu_T^{\mathrm{inh}}
 =\frac12(\phi_+)_\#\mu_T
  +\frac12(\phi_-)_\#\mu_T.
\]
Consequently,
\[
 \nu_U
 =\left(1-\frac2d\right)
   \frac{\delta_1+\delta_{-1}}2
  +\frac2d\,\nu_T^{\mathrm{inh}}.
\]
The first term is the structural contribution from the flip--flop construction;
the second is the inherited contribution. Multiplicities follow the
characteristic-polynomial factorization, so both branches are counted even
when they coincide at an endpoint.
\end{definition}

The inherited part has total weight \(2/d\) in the full unitary law.
Normalizing this part to have mass one keeps the dependence on the
discriminant visible as the degree grows.

\section{Infinite-volume spectral law}
\label{sec:infinite-volume}

\subsection{Quenched convergence on locally tree-like graphs}

Let \(\mathcal I\subset\N\) be infinite, and let
\(G_n=(V_n,E_n)\), \(n\in\mathcal I\), be finite simple \(d\)-regular
graphs with \(|V_n|=n\).  In particular, \(n\geq d+1\) and \(dn\) is
even for \(n\in\mathcal I\).  All graph-sequence limits in this section
are taken as \(n\to\infty\) through \(\mathcal I\).
Let \(B_R(G,v)\) denote the rooted ball of graph radius \(R\), and let
\(\mathbb T_d\) be the infinite \(d\)-regular tree rooted at \(o\).
We write \(\simeq\) for rooted-graph isomorphism. The graphs are locally
tree-like in the following sense: for every fixed radius, the proportion
of vertices whose neighborhoods differ from those of the regular tree
tends to zero. Thus we assume that for every fixed \(R\),
\[
 \frac1n
 \#\left\{
 v\in V_n:
 B_R(G_n,v)\not\simeq B_R(\mathbb{T}_d,o)
 \right\}
 \longrightarrow0.
\]
To treat the entire graph sequence on one probability space, place the marked
stars on the product space
\[
 \prod_{n\in\mathcal I}\prod_{v\in V_n}\Omega_{n,v},
\]
with independent copies of the common one-vertex law, and let \(T_n\) be the
resulting discriminant.  All almost-sure statements for the sequence refer
to this product environment.

On the infinite \(d\)-regular tree \(\mathbb{T}_d\), put independent copies of
the same marks.  Let \(A^{\rightarrow}(\mathbb{T}_d)\) be its directed-edge
set and define
\[
 \mathcal H_\infty=\ell^2(A^{\rightarrow}(\mathbb{T}_d)),
 \qquad
 C_\infty=\bigoplus_{v\in\mathbb{T}_d}(2q_vq_v^*-I_d),
\]
and define the flip--flop shift \(S_\infty\) and boundary operator
\(K_\infty\) as in the finite-graph case. Then
\[
 U_\infty=S_\infty C_\infty
 \quad\text{and}\quad
 T_\infty=K_\infty S_\infty K_\infty^*.
\]
The direct-sum coin and the shift are unitary, \(K_\infty K_\infty^*=I\),
and hence \(U_\infty\) is unitary while \(T_\infty\) is a bounded
self-adjoint contraction. The spectral measure at the root depends on the
random environment. Its expectation is the deterministic probability measure
\[
 \mu_{d,\alpha}(B)
 =
 \E\left[
 \left\langle
 \delta_o,\1_B(T_\infty)\delta_o
 \right\rangle
 \right].
\]
Measurability of the rooted spectral measure follows from the standard theory
of measurable bounded local operators; see
\cite{BordenaveLelarge2010,AnantharamanEtAl2021}.

\begin{theorem}[Quenched limiting spectral law]\label{thm:quenched-limit}
For almost every realization of the vertex marks,
\[
 \mu_{T_n}
 =
 \frac1n\sum_{j=1}^n\delta_{\lambda_j(T_n)}
 \Longrightarrow
 \mu_{d,\alpha}.
\]
The limit is independent of the admissible conditional phase law.
\end{theorem}

\begin{proof}
Fix \(k\geq1\); the zeroth moment is identically one.  Put
\[
 M_{n,k}=\frac1n\Tr T_n^k
\]
and
\[
 \mathcal G_{n,k}
 =
 \left\{
 v\in V_n:
 B_{\lceil k/2\rceil}(G_n,v)
 \simeq B_{\lceil k/2\rceil}(\mathbb T_d,o)
 \right\}.
\]
A closed walk of length \(k\) based at \(v\) stays within this ball. Its
edge coefficients use only the marks at vertices visited by the walk, and
each such mark includes the full outgoing star. On a tree, each unoriented
edge is traversed equally often in both directions, so all phase factors
cancel. Symmetry of the Dirichlet distribution makes the remaining
squared-amplitude law invariant under permutations of the arc labels. Thus,
for \(v\in\mathcal G_{n,k}\),
\[
 \E\langle\delta_v,T_n^k\delta_v\rangle
 =
 m_k(d,\alpha)
 :=
 \E\langle\delta_o,T_\infty^k\delta_o\rangle
 =
 \int_{-1}^1x^k\,\mu_{d,\alpha}(\dd x).
\]
Since both the diagonal entry and \(m_k(d,\alpha)\) have absolute value at
most one,
\[
 \left|\E M_{n,k}-m_k(d,\alpha)\right|
 \leq
 2\left(1-\frac{|\mathcal G_{n,k}|}{n}\right)
 \longrightarrow0.
\]

Let \(\widetilde T_n\) be the discriminant obtained by changing the mark
at a single vertex \(v\). Only the row and column indexed by \(v\) change.
Thus, for \(A=T_n-\widetilde T_n\), there is a vector \(b\) supported on
the neighbors of \(v\) such that
\[
 A=\delta_v b^*+b\delta_v^*.
\]
This decomposition gives \(\rank A\leq2\). Write \(\|\cdot\|_1\) for
the trace norm. Since
\(\|A\|\leq\|T_n\|+\|\widetilde T_n\|\leq2\), we obtain
\(\|A\|_1\leq4\). The telescoping identity
\[
 T_n^k-\widetilde T_n^k
 =\sum_{j=0}^{k-1}T_n^jA\widetilde T_n^{\,k-1-j}
\]
therefore gives
\[
 \left|
 \Tr T_n^k-\Tr \widetilde T_n^k
 \right|
 \leq
 \|T_n^k-\widetilde T_n^k\|_1
 \leq4k.
\]
McDiarmid's bounded-difference inequality
\cite[Lemma~(1.2), p.~149]{McDiarmid1989} gives, for every \(t>0\),
\[
 \Pp\left(
 |M_{n,k}-\E M_{n,k}|>t
 \right)
 \leq
 2\exp\left(-\frac{nt^2}{8k^2}\right).
\]
For each \(k\), the Borel--Cantelli lemma yields
\(M_{n,k}\to m_k(d,\alpha)\) almost surely. Taking the countable intersection
of these probability-one events gives simultaneous convergence for all
\(k\). Since every \(\mu_{T_n}\) is supported on \([-1,1]\), convergence of
all moments implies weak convergence.

The same phase cancellation shows that the limiting law does not depend on
the admissible conditional phase distribution.
\end{proof}

Define the probability measure on the unit circle
\[
 \nu_{d,\alpha}
 =
 \frac{d-2}{2d}(\delta_1+\delta_{-1})
 +\frac1d(\phi_+)_\#\mu_{d,\alpha}
 +\frac1d(\phi_-)_\#\mu_{d,\alpha},
\]
where \(\phi_\pm(x)=x\pm i\sqrt{1-x^2}\).

\begin{corollary}[Limiting law of the quantum walks]\label{cor:unitary-limit}
Let \(\nu_{U_n}\) be the empirical spectral measure of \(U_n\), normalized
by \(d|V_n|\).  Under the assumptions of
Theorem~\ref{thm:quenched-limit},
\[
 \nu_{U_n}\Longrightarrow\nu_{d,\alpha}
 \qquad\text{almost surely}.
\]
The measure \(\nu_{d,\alpha}\) is independent of the admissible phase law.
\end{corollary}

\begin{proof}
Apply Corollary~\ref{cor:empirical-law-map} and preservation of weak
convergence under the continuous maps \(\phi_+\) and \(\phi_-\).
\end{proof}

\subsection{Dirichlet closed-walk calculus}

A diagonal entry of a power of the discriminant is a sum over closed
walks. On a tree, each edge is crossed equally often in both directions,
so its phase cancels. The remaining calculation uses independence across
vertices and the Dirichlet mixed moments within each vertex.

For \(j\geq0\), write
\[
 m_j(d,\alpha)=\int_{-1}^1x^j\,\mu_{d,\alpha}(\dd x).
\]
For a rooted closed walk
\[
 \gamma=(v_0=o,v_1,\ldots,v_{2k}=o)
\]
on \(\mathbb{T}_d\), let \(V(\gamma)\) and \(E(\gamma)\) denote its traced
vertex and edge sets.  For an edge \(e=\{v,u\}\) incident to \(v\), write
\(X_{v,e}:=X_{v,u}\).  Let \(r_e(\gamma)\) be the number of traversals of
\(e\) in one fixed direction.  Since \(\gamma\) is closed and the ambient
graph is a tree, the two directional traversal counts coincide; hence \(e\)
is traversed \(2r_e(\gamma)\) times in total.  Set
\[
 R_v(\gamma)=\sum_{\substack{e\in E(\gamma)\\e\ni v}}r_e(\gamma).
\]
Consequently,
\[
 \sum_{e\in E(\gamma)}r_e(\gamma)=k,
 \qquad
 \sum_{v\in V(\gamma)}R_v(\gamma)=2k.
\]
We use the rising factorial
\[
 (a)_r=a(a+1)\cdots(a+r-1),
 \qquad (a)_0=1.
\]

\begin{theorem}[All-order moment formula]\label{thm:moment-formula}
Each rooted closed-walk sequence in \(\mathbb T_d\) contributes one term to
the sum below. The measure
\(\mu_{d,\alpha}\) is symmetric, and
\[
 m_{2k+1}(d,\alpha)=0,
\]
while
\begin{align*}
 m_{2k}(d,\alpha)
 &=
 \sum_{\substack{\gamma:o\to o\\|\gamma|=2k}}
 \prod_{v\in V(\gamma)}
 \frac{
 \displaystyle
 \prod_{\substack{e\in E(\gamma)\\e\ni v}}
 (\alpha)_{r_e(\gamma)}
 }{
 (d\alpha)_{R_v(\gamma)}
 }\\
 &=
 \sum_{\substack{\gamma:o\to o\\|\gamma|=2k}}
 \frac{
 \displaystyle\prod_{e\in E(\gamma)}
 (\alpha)_{r_e(\gamma)}^2
 }{
 \displaystyle\prod_{v\in V(\gamma)}
 (d\alpha)_{R_v(\gamma)}
 }.
\end{align*}
In particular,
\[
 m_2(d,\alpha)=\frac1d
\]
and
\[
 m_4(d,\alpha)
 =
 \frac{2\alpha(d-1)}{d^2(d\alpha+1)}
 +
 \frac{(\alpha+1)^2}{d(d\alpha+1)^2}.
\]
\end{theorem}

\begin{proof}
Let \(J\) be the parity operator
\[
 (Jf)(v)=(-1)^{\operatorname{dist}(o,v)}f(v).
\]
Every edge joins opposite parity classes, so
\(JT_\infty J=-T_\infty\), while \(J\delta_o=\delta_o\). The root spectral
measure, and hence \(\mu_{d,\alpha}\), is therefore symmetric. Bipartiteness
also rules out odd closed walks. For an even closed walk, phases cancel edge
by edge, leaving the weight
\[
 \prod_{e=\{x,y\}}
 \bigl(X_{x,y}X_{y,x}\bigr)^{r_e(\gamma)}.
\]
The vertex marks are independent.  At a fixed vertex,
the Dirichlet moment identity gives
\[
 \E\left[
 \prod_{\substack{e\in E(\gamma)\\e\ni v}}X_{v,e}^{r_e}
 \right]
 =
 \frac{\prod_{\substack{e\in E(\gamma)\\e\ni v}}(\alpha)_{r_e}}
 {(d\alpha)_{R_v(\gamma)}}.
\]
Independence across vertices lets us multiply these factors. Summing over
closed walks gives the first expression; the second follows because each
traced edge has exactly two endpoints.

At length two, there are \(d\) walks \(o\to u\to o\), each with expected
weight \(d^{-2}\). Thus \(m_2(d,\alpha)=1/d\). At length four, there are
three types. Walks \(o\to u\to o\to w\to o\) with \(u\neq w\) and
\(o\to u\to w\to u\to o\) with \(w\neq o\) each have \(d(d-1)\)
realizations, all with expected weight
\[
 \frac{\alpha}{d^3(d\alpha+1)}.
\]
The alternating type \(o\to u\to o\to u\to o\) has \(d\) realizations,
each with expected weight
\[
 \left[\frac{\alpha+1}{d(d\alpha+1)}\right]^2.
\]
Adding the three contributions gives the stated formula for \(m_4\).
\end{proof}

\subsection{A general rare-path criterion and the Dirichlet specialization}

Full spectrum can arise from rare paths on which two directions at each
interior vertex receive nearly half the mass each. The remaining mass is
small, so the path is weakly coupled to the rest of the tree. This
mechanism applies beyond the Dirichlet distribution. Let
\[
 \Delta_{d-1}
 =\left\{x\in[0,1]^d:\sum_{j=1}^d x_j=1\right\},
\]
and let \(\mathscr L\) be a permutation-invariant probability law on
\(\Delta_{d-1}\).  We say that \(\mathscr L\) satisfies the
\emph{half-split support condition} if, for every \(\eta\in(0,1/4)\),
\[
 \mathscr L\left(
 \left\{x:\left|x_1-\frac12\right|<\eta,
             \ \left|x_2-\frac12\right|<\eta\right\}
 \right)>0.
 \tag{H}
\]
Attach independent \(\mathscr L\)-distributed marks to the vertices of
\(\mathbb T_d\), and let \(T_{\mathscr L}\) be the weighted adjacency operator with
nonnegative edge weights
\[
 (T_{\mathscr L})_{u,v}=\sqrt{X_{u,v}X_{v,u}}.
\]
It is a self-adjoint contraction by the same discriminant construction as in
Section~\ref{sec:model-reduction}.  Write
\[
 \mu_{\mathscr L}(B)
 =\E\langle\delta_o,\1_B(T_{\mathscr L})\delta_o\rangle.
\]
The symmetric Dirichlet law satisfies (H). For every \(\eta>0\), the
half-split event contains a nonempty relatively open set in the simplex
interior, where the Dirichlet density is strictly positive.

For a simple path
\[
 P=(v_0,v_1,\ldots,v_L,v_{L+1})
\]
and \(0<\eta<1/4\), let \(\mathcal G(P,\eta)\) be the event that, for every
\(1\le j\le L\),
\begin{align*}
 \left|X_{v_j,v_{j-1}}-\frac12\right|&<\eta,\\
 \left|X_{v_j,v_{j+1}}-\frac12\right|&<\eta.
\end{align*}
The simplex constraint then gives
\[
 \sum_{\substack{w\sim v_j\\
 w\notin\{v_{j-1},v_{j+1}\}}}X_{v_j,w}<2\eta.
\]

\begin{lemma}[Occurrence of rare paths]\label{lem:rare-paths}
Assume (H). For every fixed \(L\ge2\) and \(\eta\in(0,1/4)\), the event
\(\mathcal G(P,\eta)\) has strictly positive probability. In any deterministic
infinite family of pairwise vertex-disjoint paths of this length, almost
surely infinitely many paths satisfy \(\mathcal G(P,\eta)\).
\end{lemma}

\begin{proof}
At each interior vertex, the two path directions correspond to distinct
coordinates. By permutation invariance and (H), the required inequalities
hold with positive probability. Let \(p_\eta>0\) be the one-vertex probability
in (H). Independence of the vertex marks gives
\[
 \Pp(\mathcal G(P,\eta))=p_\eta^L>0.
\]
For a deterministic family of pairwise vertex-disjoint paths,
the corresponding events depend on disjoint sets of marks and are therefore
independent.  The second Borel--Cantelli lemma proves the assertion.
\end{proof}

For \(L\ge1\), let \(J_L\) be the path operator with Dirichlet boundary
conditions on \(\ell^2(\{1,\ldots,L\})\),
\[
 (J_Lf)_j=\frac12 f_{j-1}+\frac12 f_{j+1},
 \qquad f_0=f_{L+1}=0.
\]
Its normalized eigenvectors and eigenvalues are
\[
 \varphi_j^{(L,r)}
 =\sqrt{\frac{2}{L+1}}
 \sin\left(\frac{r\pi j}{L+1}\right),
 \qquad
 \lambda_{L,r}
 =\cos\left(\frac{r\pi}{L+1}\right),
 \quad 1\le r\le L.
\]

\begin{lemma}[Rare-path quasimodes]\label{lem:rare-quasimode}
Suppose \(\mathcal G(P,\eta)\) occurs.  Extend
\(\varphi^{(L,r)}\) by zero from \(v_1,\ldots,v_L\) to the full tree, and
denote the resulting unit vector by \(\psi_{P,L,r}\).  Then
\[
 \left\|(T_{\mathscr L}-\lambda)\psi_{P,L,r}\right\|
 \le
 |\lambda_{L,r}-\lambda|
 +2\eta
 +\sqrt{2\eta+\frac{4}{L+1}}
\]
for every \(\lambda\in\R\).
\end{lemma}

\begin{proof}
Let \(Q_P\) be the orthogonal projection onto
\(\operatorname{span}\{\delta_{v_1},\ldots,\delta_{v_L}\}\). On the specified
good event, every internal path weight satisfies
\[
 \left|\sqrt{X_{v_j,v_{j+1}}X_{v_{j+1},v_j}}-\frac12\right|<\eta,
 \qquad 1\le j<L,
\]
because the geometric mean of two positive numbers lies between them. The
compression of \(T_{\mathscr L}\) to the path support therefore differs from
\(J_L\) by a tridiagonal matrix with maximum absolute row sum at most
\(2\eta\). Hence
\[
 \left\|Q_P(T_{\mathscr L}-\lambda)\psi_{P,L,r}\right\|
 \le 2\eta+|\lambda_{L,r}-\lambda|.
\]

An outside vertex is adjacent to at most one of
\(v_1,\ldots,v_L\), since the ambient tree has no cycles.  Moreover,
\[
 X_{v_j,w}X_{w,v_j}\le X_{v_j,w}.
\]
Writing \(\psi=\psi_{P,L,r}\), the tree structure gives
\[
 \|(I-Q_P)T_{\mathscr L}\psi\|^2
 =
 \sum_{j=1}^L|\varphi_j^{(L,r)}|^2
 \sum_{\substack{w\notin\{v_1,\ldots,v_L\}\\w\sim v_j}}
 X_{v_j,w}X_{w,v_j}.
\]
Off-path branches contribute at most \(2\eta\) to this squared norm. The two
boundary edges contribute at most
\[
 |\varphi_1^{(L,r)}|^2+|\varphi_L^{(L,r)}|^2
 \le\frac{4}{L+1},
\]
since every edge weight is at most one. Combining the bound on the path
support with the leakage bound proves the claim.
\end{proof}

\begin{theorem}[Rare-path full-spectrum criterion]\label{thm:full-spectrum}
If \(\mathscr L\) is permutation invariant and satisfies (H), then
\[
 \sigma(T_{\mathscr L})
 =\sigma_{\mathrm{ess}}(T_{\mathscr L})
 =[-1,1]
 \qquad\text{almost surely}.
\]
\end{theorem}

\begin{proof}
The contraction property gives
\(\sigma(T_{\mathscr L})\subset[-1,1]\).  Let \(D\) be a countable dense
subset of \([-1,1]\), and fix \(\lambda\in D\).  Choose
\(L_n\uparrow\infty\), indices \(1\le r_n\le L_n\), and
\(\eta_n\downarrow0\) such that
\[
 \lambda_{L_n,r_n}\longrightarrow\lambda.
\]
Choose a deterministic family \((P_{n,k})_{n,k\geq1}\) of pairwise
vertex-disjoint paths, with \(P_{n,k}\) having \(L_n+2\) vertices. Such a
family can be constructed recursively: deleting finitely many vertices from
\(\mathbb T_d\) always leaves an infinite component. For each \(n\),
Lemma~\ref{lem:rare-paths} gives infinitely many \(\eta_n\)-good paths almost
surely. Select the good path with the smallest index \(k\); this is a
measurable choice. Let \(\psi_n\) be its vector from
Lemma~\ref{lem:rare-quasimode}. The supports are disjoint, so
\(\psi_n\rightharpoonup0\), and
\[
 \|(T_{\mathscr L}-\lambda)\psi_n\|
 \le
 |\lambda_{L_n,r_n}-\lambda|
 +2\eta_n
 +\sqrt{2\eta_n+\frac{4}{L_n+1}}
 \longrightarrow0.
\]
Thus \((\psi_n)\) is a singular Weyl sequence. Since \(D\) and the indices \(n\) are countable, intersecting the
corresponding probability-one events gives a single event
on which all these sequences exist. Closedness of the essential spectrum
then yields \([-1,1]\subset\sigma_{\mathrm{ess}}(T_{\mathscr L})\) almost
surely. The reverse inclusions follow from the contraction property.
\end{proof}

\begin{theorem}[Full support of the averaged tree law]\label{thm:full-support}
Under the hypotheses of Theorem~\ref{thm:full-spectrum},
\[
 \supp\mu_{\mathscr L}=[-1,1].
\]
For the symmetric Dirichlet environment, this gives
\(\supp\mu_{d,\alpha}=[-1,1]\).  The corresponding infinite-tree quantum
walk satisfies
\[
 \sigma(U_\infty)=\T
 \qquad\text{almost surely},
\]
and the limiting unitary law \(\nu_{d,\alpha}\) has support \(\T\).
\end{theorem}

\begin{proof}
Let \(I\) be an open interval intersecting \([-1,1]\), and set
\(P_\omega(I)=\1_I(T_{\mathscr L}(\omega))\). Suppose that
\(\mu_{\mathscr L}(I)=0\). The positive-amplitude environment is invariant in
law under tree automorphisms because the mark law is permutation invariant.
Thus, for every vertex \(v\),
\[
 \E\langle\delta_v,P_\omega(I)\delta_v\rangle=0.
\]
The integrand is nonnegative.  Since the vertex set is countable, almost
surely
\[
 \langle\delta_v,P_\omega(I)\delta_v\rangle=0
 \qquad\text{for all }v.
\]
As \(P_\omega(I)\) is an orthogonal projection,
\[
 \langle\delta_v,P_\omega(I)\delta_v\rangle
 =\|P_\omega(I)\delta_v\|^2,
\]
so \(P_\omega(I)\delta_v=0\) for every basis vector, and hence
\(P_\omega(I)=0\). A zero spectral projection on an open interval means that
the interval is disjoint from the spectrum, contrary to
Theorem~\ref{thm:full-spectrum}. Every open interval intersecting \([-1,1]\)
therefore has positive \(\mu_{\mathscr L}\)-mass.

For the Dirichlet reflection walk,
\(K_\infty K_\infty^*=I\) and
\(U_\infty=S_\infty(2K_\infty^*K_\infty-I)\), so the hypotheses of the
abstract spectral mapping theorem \cite{SegawaSuzuki2019} are satisfied.
Away from the exceptional points, that theorem identifies the inherited
unitary spectrum with the inverse Joukowsky image of \((-1,1)\), namely
\(\T\setminus\{\pm1\}\). Closedness of the spectrum then gives the full unit
circle. Corollary~\ref{cor:unitary-limit} also gives
\(\supp\nu_{d,\alpha}=\T\).
\end{proof}

\begin{remark}
The rare-path theorem applies to every permutation-invariant simplex law that
satisfies the half-split condition.  The simplex constraint controls leakage
when the two path coordinates are close to one half.  For the Dirichlet law,
the explicit closed-walk factors additionally produce the scale
\(d\alpha_d^2\) analyzed in Section~\ref{sec:high-degree}.
\end{remark}

\subsection{Disorder endpoints}

Here the degree remains fixed. As \(\alpha\downarrow0\), a vertex places
nearly all its mass in one direction; as \(\alpha\to\infty\), its masses
approach the uniform vector. The next theorem identifies the spectral laws
corresponding to these two limits.

Let \(A_{\mathbb T_d}\) denote the unweighted adjacency operator on
\(\mathbb T_d\), and let \(\mu_d^{\mathrm{KM}}\) be the spectral measure
of \(A_{\mathbb T_d}/d\) at the root.  It has density
\[
 \rho_d^{\mathrm{KM}}(x)
 =
 \frac{\sqrt{4(d-1)-d^2x^2}}
 {2\pi(1-x^2)}
 \1_{\{|x|\leq 2\sqrt{d-1}/d\}}.
\]

\begin{theorem}[Strong- and weak-amplitude-disorder limits]\label{thm:disorder-limits}
For fixed \(d\geq3\),
\[
 \mu_{d,\alpha}
 \Longrightarrow
 \left(1-\frac1d\right)\delta_0
 +
 \frac1{2d}(\delta_{-1}+\delta_1)
 \qquad
 (\alpha\downarrow0),
\]
and
\[
 \mu_{d,\alpha}
 \Longrightarrow
 \mu_d^{\mathrm{KM}}
 \qquad
 (\alpha\to\infty).
\]
\end{theorem}

\begin{proof}
Fix \(k\geq1\) and consider first \(\alpha\downarrow0\).  For every
fixed integer \(r\geq1\),
\[
 (\alpha)_r\sim\alpha(r-1)!,
 \qquad
 (d\alpha)_r\sim d\alpha(r-1)!.
\]
Let the traced tree of a closed walk have \(e\) edges.  At a visited vertex
\(v\), let \(\deg_H(v)\) be the number of traced edges incident to \(v\).
The vertex factor in Theorem~\ref{thm:moment-formula} is then of order
\(\alpha^{\deg_H(v)-1}\).  Since a tree with \(e\) edges has \(e+1\)
vertices and
\(\sum_v\deg_H(v)=2e\), the whole walk contribution is of order
\(\alpha^{e-1}\). Contributions from walks tracing at least two edges are
therefore \(o(1)\). There are finitely many such multi-edge walks of length
\(2k\), so their total contribution is \(o(1)\). A one-edge walk makes
\(k\) round trips across one of the \(d\) edges incident to the root. Each
endpoint contributes
\[
 \frac{(\alpha)_k}{(d\alpha)_k}\longrightarrow\frac1d.
\]
Consequently,
\[
 m_{2k}(d,\alpha)\longrightarrow
 d\left(\frac1d\right)^2=\frac1d,
 \qquad
 m_{2k+1}(d,\alpha)=0.
\]
These are the moments of the stated matching law. To interpret the limit,
note that a symmetric Dirichlet vector converges as \(\alpha\downarrow0\)
to a uniformly chosen vertex of the simplex. Each graph vertex therefore
chooses one neighbor uniformly, and only mutually chosen edges survive.

As \(\alpha\to\infty\), for every fixed exponent vector at a visited
vertex,
\[
 \frac{\prod_{e\ni v}(\alpha)_{r_e}}
 {(d\alpha)_{R_v}}
 \longrightarrow d^{-R_v}.
\]
Since \(\sum_vR_v=2k\), every closed walk of length \(2k\) receives the
limiting weight \(d^{-2k}\), exactly its weight for
\(A_{\mathbb{T}_d}/d\).  All measures involved are supported on
\([-1,1]\), so moment convergence implies weak convergence in both endpoint
limits.
\end{proof}

\begin{corollary}[Support collapse at the disorder endpoints]\label{cor:support-collapse}
For every \(0<\alpha<\infty\), the measure \(\mu_{d,\alpha}\) has support
\([-1,1]\).  The two endpoint laws have smaller support:
\[
 \supp\left[
 \left(1-\frac1d\right)\delta_0
 +\frac1{2d}(\delta_{-1}+\delta_1)
 \right]
 =\{-1,0,1\},
\]
whereas
\[
 \supp\mu_d^{\mathrm{KM}}
 =
 \left[-\frac{2\sqrt{d-1}}d,
       \frac{2\sqrt{d-1}}d\right]
 \subsetneq[-1,1].
\]
Thus both endpoint weak limits have strictly smaller topological support than
any finite-\(\alpha\) law.
\end{corollary}

\begin{remark}[Fractional-moment structure at finite disorder]
Theorem~\ref{thm:full-spectrum} determines the almost-sure spectrum.  For
\(s\in(0,1)\), the edge coefficients have the explicit fractional moments
\[
 \E|T_{u,v}|^{2s}
 =\left(\frac{(\alpha)_s}{(d\alpha)_s}\right)^2,
\]
where \((a)_s=\Gamma(a+s)/\Gamma(a)\), the real-order extension of the
rising factorial used above.  Since the marks at \(u\) and \(v\) are
independent and \(|T_{u,v}|^2=X_{u,v}X_{v,u}\), the displayed factorization
follows directly. This identity provides an explicit edge-moment estimate
for fractional-moment analysis \cite{AizenmanMolchanov1993} in the presence
of off-diagonal disorder and vertex-wise dependence.
\end{remark}

\section{High-degree discriminant and eigenphase transitions}
\label{sec:high-degree}

We now study the high-degree asymptotics of the infinite-tree law
\(\mu_{d,\alpha}\). Since \(m_2(d,\alpha)=1/d\), rescaling spectral
coordinates by \(\sqrt d\) produces a law with second moment one.
Combined with Theorem~\ref{thm:quenched-limit}, these results describe
the large-graph limit at fixed \(d\), followed by the limit as \(d\)
tends to infinity.

For the full unitary law, Corollary~\ref{cor:unitary-limit} and the preceding
decomposition imply, uniformly in $\alpha$,
\[
 d_{\mathrm{TV}}\!\left(
 \nu_{d,\alpha},\frac{\delta_1+\delta_{-1}}2
 \right)\leq\frac2d,
\]
where
\[
 d_{\mathrm{TV}}(\mu,\nu)=\sup_{B\in\mathcal B(\C)}|\mu(B)-\nu(B)|.
\]
Thus the full limiting unitary law converges to the equally weighted
structural atoms as $d\to\infty$. The phase transition concerns the normalized
inherited law, or equivalently the discriminant law and the centered inherited
eigenphases.

To compare different degrees, we group closed walks by the finite tree
pattern they trace. This separates the number of embeddings into the
regular tree from the expected weight of each embedded walk.

Fix an integer \(k\geq1\). Let \(\mathcal{W}_k\) be the finite set of
equivalence classes of rooted closed walks of length \(2k\) on trees. Two
walks are equivalent if a root-preserving isomorphism between their traced
trees maps the vertex at each time in one walk to the vertex at the same time
in the other. To obtain a canonical representative, label the root \(0\) and
assign each newly visited vertex the smallest unused positive integer. This
orders the children of each vertex by their first visit.

For \(\gamma\in\mathcal{W}_k\), write \(H_\gamma\) for the finite rooted
tree traced by the walk and set
\[
 \edg(\gamma)=|E(H_\gamma)|.
\]
Let \(c_v(\gamma)\) be the number of children of \(v\), let
\(r_e(\gamma)\geq1\) be the number of traversals of \(e\) in one fixed
direction, and put
\[
 R_v(\gamma)=\sum_{\substack{e\in E(H_\gamma)\\e\ni v}}r_e(\gamma).
\]
Since every edge is traversed equally often in both directions,
\[
 \sum_{e\in E(H_\gamma)}r_e(\gamma)=k,
 \qquad
 \sum_{v\in V(H_\gamma)}R_v(\gamma)=2k.
\]
For the falling factorial, we write
\[
 x^{\underline r}=x(x-1)\cdots(x-r+1),
 \qquad
 x^{\underline0}=1,
\]
reserving \((a)_r\) for the rising factorial throughout the paper.  We also write
\[
 \Cat_k=\frac{1}{k+1}\binom{2k}{k}.
\]

\begin{proposition}[Grouped moment expansion]\label{prop:grouped-moments}
For \(d\geq2k\),
\[
 m_{2k}(d,\alpha)
 =
 \sum_{\gamma\in\mathcal{W}_k}
 N_d(\gamma)
 \frac{
 \displaystyle\prod_{e\in E(H_\gamma)}
 (\alpha)_{r_e(\gamma)}^2
 }{
 \displaystyle\prod_{v\in V(H_\gamma)}
 (d\alpha)_{R_v(\gamma)}
 },
\]
where
\[
 N_d(\gamma)
 =
 d^{\underline{c_o(\gamma)}}
 \prod_{\substack{v\in V(H_\gamma)\\v\neq o}}
 (d-1)^{\underline{c_v(\gamma)}}.
\]
\end{proposition}

\begin{proof}
The sufficient condition \(d\geq2k\) ensures that all ordered embeddings
below exist and all falling factorials are nonnegative. Fix the canonical
representative of \(\gamma\). The first-visit labels distinguish its vertices
and the children of each vertex. To embed \(H_\gamma\) into
\(\mathbb{T}_d\) while fixing the root, choose an ordered injection of the
\(c_o(\gamma)\) root children into the \(d\) available neighbors. There are
\(d^{\underline{c_o(\gamma)}}\) choices. At a nonroot vertex, the parent's
image is fixed, leaving \(d-1\) neighbors for its children. Multiplying these
counts gives \(N_d(\gamma)\). The all-order moment formula now gives the
result, since
\[
 \prod_{v\in V(H_\gamma)}
 \prod_{\substack{e\in E(H_\gamma)\\e\ni v}}
 (\alpha)_{r_e(\gamma)}
 =
 \prod_{e\in E(H_\gamma)}(\alpha)_{r_e(\gamma)}^2.
\]
\end{proof}

Define
\[
 \varepsilon_{d,\alpha}
 =
 \frac{(1+\alpha)^2}{d\alpha^2}
 =
 \frac{(1+\alpha^{-1})^2}{d}.
\]

The leading classes traverse each edge once in each direction; their
number is Catalan. The next bound controls both their deviation from unit
weight after rescaling and the contribution of walks that repeat edges.

\begin{lemma}[Catalan sector and a uniform remainder bound]
\label{lem:high-degree-bound}
For every fixed \(k\geq1\), there exists \(C_k<\infty\), depending only on
\(k\), such that whenever \(d\geq2k\) and
\(\varepsilon_{d,\alpha}\leq1\),
\[
 \left|
 d^k m_{2k}(d,\alpha)-\Cat_k
 \right|
 \leq
 C_k\varepsilon_{d,\alpha}.
\]
\end{lemma}

\begin{proof}
For \(\gamma\in\mathcal{W}_k\), let \(A_{d,\alpha}(\gamma)\) denote its
summand in Proposition~\ref{prop:grouped-moments}, and write
\(\ell=\edg(\gamma)\).

First suppose that \(\ell<k\).  Since the total number of children in
\(H_\gamma\) is \(\ell\),
\[
 N_d(\gamma)\leq d^\ell.
\]
Moreover,
\[
 (d\alpha)_{R_v}\geq(d\alpha)^{R_v}.
\]
For \(1\leq r\leq k\),
\begin{align*}
 (\alpha)_r
 &=
 \alpha^r\prod_{j=1}^{r-1}
 \left(1+\frac{j}{\alpha}\right)\\
 &\leq
 \alpha^r
 \bigl[k(1+\alpha^{-1})\bigr]^{r-1}.
\end{align*}
Using
\(\sum_f r_f=k\) and \(\sum_vR_v=2k\), we obtain
\[
 0\leq
 d^kA_{d,\alpha}(\gamma)
 \leq
 k^{2(k-\ell)}
 \varepsilon_{d,\alpha}^{k-\ell}.
\]
Since \(\mathcal{W}_k\) is finite and
\(\varepsilon_{d,\alpha}\leq1\), the sum of all terms with \(\ell<k\) is at
most \(C_k^{(1)}\varepsilon_{d,\alpha}\), for a constant depending only on
\(k\).

Now suppose that \(\ell=k\). Then \(r_f=1\) for every edge
\(f\in E(H_\gamma)\). Each step away from the root discovers a new edge;
each step toward the root retraces the most recently discovered edge not yet
retraced. The distance from the root is therefore a Dyck path. The
first-discovery convention gives a bijection with Dyck paths of length
\(2k\), so this sector contains exactly \(\Cat_k\) classes.

For a class in this sector, the root contribution satisfies
\[
 \frac{d^{\underline{c_o}}\alpha^{c_o}}{(d\alpha)_{c_o}}
 =\prod_{j=1}^{c_o-1}\frac{(d-j)\alpha}{d\alpha+j},
\]
whereas a nonroot vertex with \(c_v\) children contributes
\[
 \frac{(d-1)^{\underline{c_v}}\alpha^{c_v+1}}
 {(d\alpha)_{c_v+1}}
 =\frac1d\prod_{j=1}^{c_v}\frac{(d-j)\alpha}{d\alpha+j}.
\]
There are \(k\) nonroot vertices, so their prefactors cancel \(d^k\). Thus
\begin{align*}
 d^kA_{d,\alpha}(\gamma)
 &=
 \prod_{j=1}^{c_o(\gamma)-1}
 \frac{(d-j)\alpha}{d\alpha+j}
 \\
 &\quad\times
 \prod_{v\neq o}
 \prod_{j=1}^{c_v(\gamma)}
 \frac{(d-j)\alpha}{d\alpha+j}.
\end{align*}
There are
\[
 c_o(\gamma)-1+\sum_{v\neq o}c_v(\gamma)=k-1
\]
nontrivial factors.  Each belongs to \([0,1]\), and for every index
appearing above,
\[
 0\leq
 1-\frac{(d-j)\alpha}{d\alpha+j}
 =
 \frac{j(\alpha+1)}{d\alpha+j}
 \leq
 k\frac{1+\alpha^{-1}}{d}
 \leq
 k\varepsilon_{d,\alpha}.
\]
The inequality
\(1-\prod_i(1-u_i)\leq\sum_i u_i\), valid for \(u_i\in[0,1]\), therefore
shows that
\[
 \left|d^kA_{d,\alpha}(\gamma)-1\right|
 \leq
 k^2\varepsilon_{d,\alpha}.
\]
Sum over the \(\Cat_k\) Catalan classes and combine this bound with the
non-Catalan estimate to conclude the proof.
\end{proof}

Let
\[
 \mu_{\mathrm{sc}}(\mathrm{d}x)
 =
 \frac{1}{2\pi}\sqrt{4-x^2}\,
 \1_{\{|x|\leq2\}}\,\mathrm{d}x
\]
be the standard semicircle law.

\begin{theorem}[High-degree semicircle regime]
\label{thm:semicircle}
Let \(d\to\infty\) and let \(\alpha_d>0\).  If
\[
 d\alpha_d^2\longrightarrow\infty,
\]
then
\[
 (\sqrt d\,\cdot)_\#\mu_{d,\alpha_d}
 \Longrightarrow
 \mu_{\mathrm{sc}}.
\]
More precisely, for every fixed \(k\geq1\) and all sufficiently large
\(d\),
\[
 d^k m_{2k}(d,\alpha_d)
 =
 \Cat_k
 +O_k\left(
 \frac{(1+\alpha_d)^2}{d\alpha_d^2}
 \right),
\]
whereas all odd moments vanish.  In particular, for fixed \(\alpha>0\),
the error is \(O_{k,\alpha}(d^{-1})\).
\end{theorem}

\begin{proof}
Because \(d\to\infty\),
\[
 d\alpha_d^2\to\infty
 \quad\Longleftrightarrow\quad
 \varepsilon_{d,\alpha_d}\to0.
\]
Indeed,
\[
 \varepsilon_{d,\alpha_d}
 =
 \frac1d+\frac{2}{d\alpha_d}+\frac{1}{d\alpha_d^2},
\]
and \(d\alpha_d=\sqrt d\sqrt{d\alpha_d^2}\to\infty\) whenever the condition
on the left holds. Lemma~\ref{lem:high-degree-bound} gives the stated error
bound and convergence of all even moments to the Catalan numbers.

Let
\[
 \widehat\mu_d=(\sqrt d\,\cdot)_\#\mu_{d,\alpha_d}.
\]
Since
\[
 \int x^2\,\widehat\mu_d(\mathrm{d}x)
 =d m_2(d,\alpha_d)=1,
\]
the family \((\widehat\mu_d)\) is tight. Consider a subsequence converging
weakly to \(\nu\). For each fixed \(k\), apply the moment estimate at index
\(k+1\) to obtain
\[
 \sup_d\int |x|^{2k+2}\,\widehat\mu_d(\mathrm{d}x)<\infty
\]
after discarding finitely many indices.  Hence \((|x|^{2k})\) is uniformly
integrable along the subsequence, and
\[
 \int x^{2k}\,\nu(\mathrm{d}x)=\Cat_k.
\]
Symmetry gives the vanishing odd moments. Since the semicircle law is
moment-determinate, \(\nu=\mu_{\mathrm{sc}}\). Every convergent subsequence
has this limit, which proves the theorem.
\end{proof}

The second moment is fixed by the normalization and cannot distinguish
the high-degree regimes. The fourth moment already detects the scale
that leads to the semicircle law.

\begin{proposition}[Fourth-moment characterization of the scale]
\label{prop:fourth-moment-scale}
Let \(d\to\infty\) and \(\alpha_d>0\).  Then
\[
 d^2m_4(d,\alpha_d)\longrightarrow2
 \quad\Longleftrightarrow\quad
 d\alpha_d^2\longrightarrow\infty.
\]
Thus the scale in Theorem~\ref{thm:semicircle} is necessary and sufficient
for convergence of the rescaled fourth moment to its semicircle value.
\end{proposition}

\begin{proof}
The exact fourth-moment formula gives
\[
 d^2m_4(d,\alpha)-2
 =
 \frac{(\alpha+1)\{d(1-\alpha)-2\}}
 {(d\alpha+1)^2}.
\]
The implication from right to left also follows from
Theorem~\ref{thm:semicircle}. For the converse, suppose that \(d\alpha_d^2\)
does not tend to infinity. It is then bounded by a finite constant \(M\)
along a subsequence, on which \(\alpha_d\leq\sqrt{M/d}\to0\). Eventually,
\(\alpha_d\leq1/2\) and \(d\geq8\), so
\[
 d(1-\alpha_d)-2\geq\frac d4.
\]
Moreover,
\[
 (d\alpha_d+1)^2
 \leq
 2d^2\alpha_d^2+2
 \leq
 2d(M+1).
\]
Therefore, along this subsequence,
\[
 d^2m_4(d,\alpha_d)-2
 \geq
 \frac{1}{8(M+1)},
\]
which rules out convergence to zero.  In particular, if
\(d\alpha_d^2\to\tau\in(0,\infty)\), then the same exact formula gives
\(d^2m_4(d,\alpha_d)\to2+\tau^{-1}\), in agreement with the critical
moment formula below.
\end{proof}

At a finite positive limit of \(d\alpha_d^2\), walks that repeat edges
continue to contribute. Their limiting weights define the following
family of non-semicircular laws.

\begin{theorem}[Critical high-degree family]
\label{thm:critical-family}
Let \(d\to\infty\), and suppose that
\[
 d\alpha_d^2\longrightarrow\tau\in(0,\infty).
\]
Then there exists a symmetric probability measure \(\mu_\tau\), uniquely
determined by its moments, such that
\[
 (\sqrt d\,\cdot)_\#\mu_{d,\alpha_d}
 \Longrightarrow
 \mu_\tau,
\]
and all moments converge. The limiting odd moments vanish, and the even
moments are
\[
 M_{2k}(\tau)
 =
 \sum_{\gamma\in\mathcal{W}_k}
 \tau^{\edg(\gamma)-k}
 \prod_{e\in E(H_\gamma)}
 \bigl((r_e(\gamma)-1)!\bigr)^2.
\]
The measure \(\mu_\tau\) has unbounded support.
\end{theorem}

\begin{proof}
For each fixed moment order, take \(d\) large enough that \(d\geq2k\), as
required by the grouped expansion. The assumption implies
\[
 \alpha_d\to0,
 \qquad
 d\alpha_d=\sqrt d\sqrt{d\alpha_d^2}\to\infty.
\]
For a fixed class \(\gamma\in\mathcal{W}_k\), its contribution to the
\(2k\)-th moment of the rescaled measure can be written as
\begin{align*}
 &\frac{N_d(\gamma)}{d^{\edg(\gamma)}}
 (d\alpha_d^2)^{\edg(\gamma)-k}
 \\
 &\quad\times
 \frac{
 \displaystyle\prod_{e\in E(H_\gamma)}
 \left[(\alpha_d)_{r_e(\gamma)}/\alpha_d\right]^2
 }{
 \displaystyle\prod_{v\in V(H_\gamma)}
 \left[(d\alpha_d)_{R_v(\gamma)}/
       (d\alpha_d)^{R_v(\gamma)}\right]
 }.
\end{align*}
For every fixed positive integer \(r\),
\[
 \frac{N_d(\gamma)}{d^{\edg(\gamma)}}\longrightarrow1,
 \qquad
 \frac{(d\alpha_d)_r}{(d\alpha_d)^r}\longrightarrow1,
 \qquad
 \frac{(\alpha_d)_r}{\alpha_d}\longrightarrow(r-1)!.
\]
Since \(\mathcal{W}_k\) is finite, summation proves convergence of the
\(2k\)-th moments to \(M_{2k}(\tau)\); the odd moments vanish identically.

We next establish moment determinacy. Ordering children by first discovery
associates a rooted plane tree with each canonical walk class. Fix such a
tree with \(e\) edges. For each compatible walk, label the edges in their
global order of first discovery and record each traversal's label and
orientation relative to the root. This encoding is injective. To recover
the walk, start at the root. At the first occurrence of a label, attach a
new child to the current vertex and record the endpoints of that edge.
Every later occurrence traverses the recorded edge in the indicated
direction. Thus the word determines the canonical walk, without a separate
choice of edge labels. For fixed positive multiplicities
\(r_1+\cdots+r_e=k\), the word contains \(r_j\) copies of each orientation
of edge \(j\). The number of compatible classes is therefore at most
\[
 \frac{(2k)!}{\prod_{j=1}^e(r_j!)^2}.
\]
Multiplying by the factorial part of the critical weight gives
\[
 \frac{(2k)!}{\prod_{j=1}^e(r_j!)^2}
 \prod_{j=1}^e((r_j-1)!)^2
 \leq
 (2k)!.
\]
Rooted plane trees with \(e\) edges are counted by \(\Cat_e\), and there are
\(\binom{k-1}{e-1}\) positive multiplicity vectors.  Since
\(\Cat_e\leq4^e\),
\begin{align*}
 M_{2k}(\tau)
 &\leq
 (2k)!
 \sum_{e=1}^k
 \Cat_e\binom{k-1}{e-1}\tau^{e-k}
 \\
 &\leq
 4(2k)!\left(4+\tau^{-1}\right)^{k-1}.
\end{align*}
Together with \((2k)!\leq(2k)^{2k}\), this gives a constant \(c_\tau>0\)
such that
\[
 M_{2k}(\tau)^{-1/(2k)}\geq\frac{c_\tau}{k}.
\]
Hence
\[
 \sum_{k\geq1}M_{2k}(\tau)^{-1/(2k)}=\infty,
\]
so Carleman's condition holds.

Set \(\widehat\mu_d=(\sqrt d\,\cdot)_\#\mu_{d,\alpha_d}\). The common second
moment is one, so this family is tight. Consider any weakly convergent
subsequence with limit \(\nu\). For each fixed integer \(p\geq0\), the
convergence of the \((2p+2)\)-th moments proved above gives
\[
 \sup_d\int |x|^{2p+2}\,\widehat\mu_d(\mathrm{d}x)<\infty
\]
along the subsequence, after finitely many terms are discarded. Thus
\((|x|^{2p})\) is uniformly integrable, and \(\nu\) has moments
\(M_j(\tau)\). Carleman's condition determines the measure uniquely. Every
weakly convergent subsequence therefore has the same limit \(\mu_\tau\), so
the full sequence converges weakly and in moments. The vanishing odd moments
of the approximating measures also give symmetry.

Finally, the unique class supported on one edge contributes
\[
 \tau^{1-k}((k-1)!)^2
\]
to \(M_{2k}(\tau)\).  Therefore
\[
 M_{2k}(\tau)^{1/(2k)}
 \geq
 \tau^{(1-k)/(2k)}((k-1)!)^{1/k}
 \longrightarrow\infty.
\]
If \(\mu_\tau\) were supported on \([-R,R]\), these roots would be at most
\(R\). Its support is therefore unbounded.
\end{proof}

\subsection{Poisson-weighted tree representation at criticality}

We now give an operator representation of the critical laws. The Poisson
intensity of the squared edge weights reproduces the factorial factors in
their closed-walk moments. Because each vertex has infinitely many
children, we first establish self-adjointness and measurability, and then
identify the expected root spectral measure. The resulting tree also gives
a recursive equation for its Green function. Let \(K_0\) denote the
modified Bessel function of the second kind.

\begin{proposition}[Bessel L\'{e}vy intensity]\label{prop:bessel-intensity}
For \(\tau>0\), define the measure
\[
 \mathsf{L}_\tau(\mathrm{d}y)
 =\frac{2\tau}{y}K_0(2\sqrt{\tau y})\,\mathrm{d}y,
 \qquad y>0.
\]
Then, for every real \(r>0\),
\[
 \int_0^\infty y^r\,\mathsf{L}_\tau(\mathrm{d}y)
 =\tau^{1-r}\Gamma(r)^2.
\]
In particular,
\[
 \int y\,\mathsf{L}_\tau(\mathrm{d}y)=1,
 \qquad
 \int \sqrt y\,\mathsf{L}_\tau(\mathrm{d}y)=\pi\sqrt\tau.
\]
The measure has infinite total mass, is finite on \((\varepsilon,\infty)\)
for every \(\varepsilon>0\), and satisfies
\(\int(1\wedge y)\,\mathsf{L}_\tau(\mathrm{d}y)<\infty\).
\end{proposition}

\begin{proof}
With \(x=2\sqrt{\tau y}\),
\begin{align*}
 \int_0^\infty y^r\,\mathsf{L}_\tau(\mathrm{d}y)
 &=(4\tau)^{1-r}
   \int_0^\infty x^{2r-1}K_0(x)\,\mathrm{d}x\\
 &=\tau^{1-r}\Gamma(r)^2,
\end{align*}
where we used the standard modified-Bessel integral
\cite[Eq.~10.43.19]{NISTDLMF},
\[
 \int_0^\infty x^{\mu-1}K_0(x)\,\mathrm{d}x
 =2^{\mu-2}\Gamma(\mu/2)^2,
 \qquad \mu>0.
\]
As \(y\downarrow0\), the density of \(\mathsf{L}_\tau\) is
\(\tau y^{-1}\log(1/y)+O_\tau(y^{-1})\).  Hence its total mass is infinite,
whereas \(\int_0^1 y\,\mathsf{L}_\tau(\mathrm{d}y)<\infty\) and
\(\int_0^1 \sqrt y\,\mathsf{L}_\tau(\mathrm{d}y)<\infty\).
The exponential decay of \(K_0\) gives finiteness on
\((\varepsilon,\infty)\) and completes the L\'{e}vy-integrability assertions.
\end{proof}

Let
\[
 \mathcal V=\bigcup_{m\ge0}\N^m,
 \qquad \N^0=\{\varnothing\},
\]
be the Ulam--Harris tree. Independently at each \(v\in\mathcal V\), let
\[
 \Pi_v=\sum_{j\ge1}\delta_{Y_{v,j}}
\]
be a Poisson point process on \((0,\infty)\) with intensity
\(\mathsf{L}_\tau\). We have \(\mathsf{L}_\tau((\varepsilon,\infty))<\infty\)
for every \(\varepsilon>0\), and \(\mathsf L_\tau\) is diffuse. The process
is therefore simple almost surely, and its atoms admit a strictly decreasing
enumeration,
\[
 Y_{v,1}>Y_{v,2}>\cdots\downarrow0.
\]
Join \(v\) to \(vj\), and assign that edge the weight
\(w_{v,vj}=\sqrt{Y_{v,j}}\).  Put
\[
 Q_v=\sum_{j\ge1}Y_{v,j}.
\]
By Campbell's formula \cite{LastPenrose2017} and
Proposition~\ref{prop:bessel-intensity}, \(\E Q_v=1\), so \(Q_v<\infty\)
almost surely. Every row thus has finite squared \(\ell^2\)-norm:
\(Q_\varnothing\) at the root and \(w_{\operatorname{par}(v),v}^2+Q_v\) at
a nonroot vertex. Let \(c_{00}(\mathcal V)\) denote the finitely supported vectors. The formal
weighted adjacency matrix maps \(c_{00}(\mathcal V)\) into
\(\ell^2(\mathcal V)\). It defines a densely
defined symmetric minimal operator, which we denote by \(A_{\tau,\min}\).

\begin{lemma}[Finite-exhaustion criterion for essential self-adjointness]
\label{lem:esa-exhaustion}
Let \(\mathcal T=(V,E)\) be a countable tree, and let \(A_{\min}\) be the
weighted adjacency operator on \(c_{00}(V)\), with real symmetric weights
\(w_{u,v}=w_{v,u}\) and square-summable rows.  Suppose that there are
\(\kappa<\infty\) and increasing connected finite sets \(S_n\uparrow V\) such
that
\[
 \sum_{u\notin S_n:\,u\sim v}|w_{u,v}|^2\leq\kappa
 \qquad(n\geq1,\ v\in S_n).
\]
Then \(A_{\min}\) is essentially self-adjoint.
\end{lemma}

\begin{proof}
This is the criterion in \cite[Lemma~A.3]{BordenaveCaputoChafai2011}. We recall
its proof. It suffices to show that both deficiency spaces at
\(\pm\mathrm{i}\) are trivial. Suppose that
\(\varphi\in\operatorname{Dom}(A_{\min}^*)\) and
\(A_{\min}^*\varphi=\mathrm{i}\varphi\). Testing against coordinate vectors
and using square-summability of the rows gives, for every \(v\in V\),
\[
 \mathrm{i}\varphi(v)=\sum_{u\sim v}w_{v,u}\varphi(u).
\]
For a finite set \(S\subset V\), multiplying the deficiency equation by
\(\overline{\varphi(v)}\), summing over \(v\in S\), and taking imaginary
parts gives the discrete Green identity
\[
 \sum_{v\in S}|\varphi(v)|^2
 =
 \operatorname{Im}
 \sum_{\substack{v\in S,\ u\notin S\\u\sim v}}
 w_{v,u}\overline{\varphi(v)}\varphi(u),
\]
because the contributions from edges internal to \(S\) sum to a real number.
Hence
\[
 \sum_{v\in S}|\varphi(v)|^2
 \leq
 \sum_{v\in S}|\varphi(v)|
 \sum_{u\notin S:\,u\sim v}|w_{v,u}|\,|\varphi(u)|.
\]
Take \(S=S_n\).  By two applications of Cauchy--Schwarz and the assumed
boundary bound,
\[
 \sum_{v\in S_n}|\varphi(v)|^2
 \leq
 \sqrt{\kappa}
 \left(\sum_{v\in S_n}|\varphi(v)|^2\right)^{1/2}
 \left(\sum_{u\notin S_n}|\varphi(u)|^2\right)^{1/2}.
\]
Connectedness of \(S_n\) and the tree property ensure that each outside
vertex has at most one neighbor in \(S_n\). If the left-hand side is nonzero,
divide by its square root and square the result to obtain
\[
 \sum_{v\in S_n}|\varphi(v)|^2
 \leq
 \kappa\sum_{u\notin S_n}|\varphi(u)|^2.
\]
As \(S_n\uparrow V\), the right-hand side tends to zero and the left-hand
side increases to \(\|\varphi\|_2^2\).  Hence \(\varphi=0\).  The case
\(A_{\min}^*\varphi=-\mathrm{i}\varphi\) is identical.
\end{proof}

\begin{lemma}[Self-adjoint Poisson tree operator]\label{lem:poisson-esa}
The minimal operator \(A_{\tau,\min}\) is almost surely essentially
self-adjoint.
\end{lemma}

\begin{proof}
For every vertex \(v\), set
\[
 S_v=\sum_{j\ge1}\sqrt{Y_{v,j}}.
\]
Campbell's formula and Proposition~\ref{prop:bessel-intensity} give
\(\E S_v=\pi\sqrt\tau<\infty\), and hence \(S_v<\infty\) almost surely.
The Poisson second-moment formula gives
\[
 \E S_v^2
 =\left(\int\sqrt y\,\mathsf{L}_\tau(\mathrm{d}y)\right)^2
  +\int y\,\mathsf{L}_\tau(\mathrm{d}y)
 =\pi^2\tau+1.
\]
For \(\kappa>0\), define the integer-valued tail index
\[
 N_\kappa(v)
 =\min\left\{t\in\N\cup\{0\}:
       \sum_{j>t}Y_{v,j}\leq\kappa\right\}.
\]
Because the sequence \((\sqrt{Y_{v,j}})_j\) is nonincreasing,
\(\sqrt{Y_{v,j}}\leq S_v/j\).  Hence
\[
 \sum_{j>t}Y_{v,j}
 \leq S_v^2\sum_{j>t}\frac1{j^2}
 \leq\frac{S_v^2}{t},
 \qquad t\geq1.
\]
It follows that
\[
 N_\kappa(v)
 \leq
 \1_{\{Q_v>\kappa\}}
 \left(1+\left\lceil\frac{S_v^2}{\kappa}\right\rceil\right),
\]
and therefore
\[
 \E N_\kappa(v)
 \leq
 2\Pp(Q_v>\kappa)+\frac{\E S_v^2}{\kappa}
 \longrightarrow0
 \qquad(\kappa\to\infty).
\]
Choose \(\kappa\) such that \(\E N_\kappa(v)<1\), and retain the first
\(N_\kappa(v)\) child edges at each vertex. From any fixed vertex, the retained
descendants form a Galton--Watson tree with offspring law \(N_\kappa(v)\).
The offspring mean is less than one, so the retained tree is finite almost
surely. Countability of the
vertex set gives a single probability-one event on which this holds for every
starting vertex. The closure of any finite connected set containing the root
under retained child edges is therefore finite. Enumerate \(\mathcal V\).
Successively adjoin the path to the next vertex and take the retained-edge
closure. This constructs an increasing connected finite exhaustion \((S_n)\)
of \(\mathcal V\).

If \(v\in S_n\) is not the root, its parent belongs to \(S_n\).
Moreover, all of its first \(N_\kappa(v)\) children belong to \(S_n\).
Consequently, for every \(v\in S_n\),
\[
 \sum_{u\notin S_n:\,u\sim v}|(A_{\tau,\min})_{u,v}|^2
 =\sum_{j:\,vj\notin S_n}Y_{v,j}
 \leq\sum_{j>N_\kappa(v)}Y_{v,j}
 \leq\kappa.
\]
This is the subcritical Galton--Watson exhaustion of
\cite[Proposition~A.2]{BordenaveCaputoChafai2011}. The boundary estimate
verifies the hypotheses of Lemma~\ref{lem:esa-exhaustion}, which proves the
claim.
\end{proof}

To relate closed-walk sums to powers of the operator, we need to recognize
which vectors belong to its domain. On the event of
Lemma~\ref{lem:poisson-esa}, define the maximal formal realization by
\[
 \operatorname{Dom}(A_{\tau,\max})
 =\left\{f\in\ell^2(\mathcal V):
   \left(\sum_{u\sim v}w_{v,u}f(u)\right)_{v\in\mathcal V}
   \in\ell^2(\mathcal V)\right\}.
\]
The coordinate sums are absolutely convergent by Cauchy--Schwarz and the
square-summability of the rows.

\begin{lemma}[Minimal and maximal realizations]\label{lem:poisson-maximal}
Almost surely,
\[
 A_{\tau,\min}^*=A_{\tau,\max}.
\]
Consequently, if \(A_\tau\) denotes the self-adjoint closure of
\(A_{\tau,\min}\), then \(A_\tau=A_{\tau,\max}\).
\end{lemma}

\begin{proof}
Let \(g\in\operatorname{Dom}(A_{\tau,\min}^*)\), and write
\(h=A_{\tau,\min}^*g\).  Testing the adjoint identity against
\(\delta_v\) gives
\[
 h(v)=\sum_{u\sim v}w_{v,u}g(u).
\]
Cauchy--Schwarz gives absolute convergence of the series. Since \(h\in\ell^2\),
we have \(g\in\operatorname{Dom}(A_{\tau,\max})\). Conversely, take
\(g\in\operatorname{Dom}(A_{\tau,\max})\) and finitely supported \(f\).
Cauchy--Schwarz again gives absolute convergence of each inner sum over
neighbors. The outer sum over \(\operatorname{supp}f\) is finite, so symmetry
of the weights yields
\begin{align*}
 \langle A_{\tau,\min}f,g\rangle
 &=\sum_{u\in\operatorname{supp}f}
   \overline{f(u)}\sum_{v\sim u}w_{u,v}g(v)\\
 &=\langle f,A_{\tau,\max}g\rangle.
\end{align*}
Thus \(g\in\operatorname{Dom}(A_{\tau,\min}^*)\), proving the first
identity.  Essential self-adjointness then yields
\(A_\tau=\overline{A_{\tau,\min}}=A_{\tau,\min}^*=A_{\tau,\max}\).
\end{proof}

To establish measurability of the root spectral measure and Green
function, we approximate the operator by finite principal compressions.
For \(N\geq1\), let
\[
 \mathcal V_N
 =\left\{(v_1,\ldots,v_m):0\leq m\leq N,
       \ 1\leq v_i\leq N\right\},
\]
and let \(P_N\) project onto \(\ell^2(\mathcal V_N)\). Denote the finite
principal compression of \(A_\tau\) by \(A_{\tau,N}\), and its zero extension
to \(\ell^2(\mathcal V)\) by \(\widetilde A_{\tau,N}=P_NA_\tau P_N\).
For \(f\in c_{00}(\mathcal V)\), once \(\supp f\subset\mathcal V_N\),
we have \(\widetilde A_{\tau,N}f=P_NA_\tau f\to A_\tau f\) in \(\ell^2\).
For \(z\in\C\setminus\R\), the resolvent identity gives
\begin{align*}
 &(\widetilde A_{\tau,N}-z)^{-1}(A_\tau-z)f-f\\
 &\qquad=(\widetilde A_{\tau,N}-z)^{-1}
          (A_\tau-\widetilde A_{\tau,N})f\longrightarrow0.
\end{align*}
Since \(c_{00}(\mathcal V)\) is a core for \(A_\tau\),
\((A_\tau-z)c_{00}(\mathcal V)\) is dense. The uniform resolvent bound
\(|\Imag z|^{-1}\) therefore gives strong-resolvent convergence of
\(\widetilde A_{\tau,N}\) to \(A_\tau\).

Define the root spectral measure by
\[
 \rho_{\tau,o}(B)
 =\langle\delta_\varnothing,\1_B(A_\tau)\delta_\varnothing\rangle.
\]
The root belongs to every \(\mathcal V_N\), so the finite compression
and its zero extension have the same root spectral measure. This measure
depends measurably on finitely many edge weights. The measures converge
weakly to \(\rho_{\tau,o}\), so
\(\rho_{\tau,o}\) is a random probability measure. Their Stieltjes transforms
are measurable random elements of \(\operatorname{Hol}(\C_+)\). The resolvent
bound and Vitali's theorem upgrade pointwise convergence to locally uniform
convergence to \(z\mapsto\langle\delta_\varnothing,(A_\tau-z)^{-1}
\delta_\varnothing\rangle\). The limiting analytic function is therefore
measurable in the compact-open topology. On the common null exceptional
event, assign these objects the values associated with the zero operator.

\begin{theorem}[Poisson-tree representation of the critical law]
\label{thm:poisson-critical}
For every \(\tau>0\),
\[
 \mu_\tau(B)=\E\rho_{\tau,o}(B),
 \qquad B\in\mathcal B(\R).
\]
Thus the critical high-degree law is the expected root spectral measure of
an explicitly weighted Poisson tree.
\end{theorem}

\begin{proof}
For each \(k\geq0\), use the finite principal compressions defined above and
extend
\[
 F_{k,N}=A_{\tau,N}^k\delta_\varnothing
\]
by zero to \(\mathcal V\).  Since all edge weights are nonnegative, the
coordinates of \(F_{k,N}\) increase with \(N\).  Define the formal
\(k\)-step walk vector by
\[
 F_k(v)=\lim_{N\to\infty}F_{k,N}(v).
\]
Its coordinate at \(v\) is the sum of the weights of all length-\(k\) walks
from \(\varnothing\) to \(v\).

For \(k\geq1\) and every \(N\), pair two \(k\)-step walks with the same
endpoint and reverse the second. This identifies \(\|F_{k,N}\|_2^2\) with
the total weight of length-\(2k\) closed walks at the root in
\(\mathcal V_N\). At each visited vertex, the canonical tree's children are
ordered by first discovery, so assigning them to distinct Poisson atoms is
an ordered injection. Monotone convergence and the multivariate Mecke formula
(the Poisson factorial moment formula \cite[Chapter~4]{LastPenrose2017}) give
\begin{align*}
 \E\|F_k\|_2^2
 &=\sum_{\gamma\in\mathcal W_k}
   \prod_{e\in E(H_\gamma)}
   \int_0^\infty y^{r_e(\gamma)}
     \mathsf{L}_\tau(\mathrm{d}y)\\
 &=\sum_{\gamma\in\mathcal W_k}
   \tau^{\edg(\gamma)-k}
   \prod_{e\in E(H_\gamma)}
   \bigl((r_e(\gamma)-1)!\bigr)^2
 =M_{2k}(\tau).
\end{align*}
This expectation is finite. Together with \(F_0=\delta_\varnothing\), it gives
\(F_k\in\ell^2(\mathcal V)\) for every \(k\geq0\) on a common probability-one
event.

For every \(k\geq0\), Cauchy--Schwarz gives absolute convergence of the
coordinate sums in \(A_{\tau,\max}F_k=F_{k+1}\), and the monotone walk
expansion proves this identity. Since \(F_{k+1}\in\ell^2\),
Lemma~\ref{lem:poisson-maximal} yields
\(F_k\in\operatorname{Dom}(A_\tau)\) and
\[
 A_\tau F_k=F_{k+1}.
\]
Induction yields
\[
 \delta_\varnothing\in\operatorname{Dom}(A_\tau^k),
 \qquad
 A_\tau^k\delta_\varnothing=F_k.
\]
Consequently,
\[
 \E\int_\R x^{2k}\,\rho_{\tau,o}(\mathrm{d}x)
 =\E\|A_\tau^k\delta_\varnothing\|_2^2
 =M_{2k}(\tau).
\]
Bipartite symmetry makes all odd moments vanish. Thus the probability
measures \(\E\rho_{\tau,o}\) and \(\mu_\tau\) have the same moments. The
Carleman estimate in Theorem~\ref{thm:critical-family} implies moment
determinacy, so the two measures coincide.
\end{proof}

\begin{remark}[Poisson-tree recursion]
The Poisson-tree representation in Theorem~\ref{thm:poisson-critical} links the
critical law to a recursive operator structure. Removing the root separates
independent copies of the same rooted tree, yielding the cavity equation below.
\end{remark}

The recursion must determine the law of an entire Green function, not
just its value at one spectral parameter. We therefore work with random
analytic functions and compare their joint values at finitely many points.
Let
\begin{equation}\label{eq:herglotz-class}
 \mathfrak H
 =\left\{h\in\operatorname{Hol}(\C_+,\C_+):
   |h(z)|\leq(\Imag z)^{-1}\ \text{for all }z\in\C_+\right\},
\end{equation}
with the topology of locally uniform convergence. The ambient space
\(\operatorname{Hol}(\C_+)\) is Polish. Its subset \(\mathfrak H\) is Borel,
and hence a standard Borel space. Indeed, nonnegativity of the imaginary
part and the resolvent bound can be checked on a countable dense subset
of \(\C_+\). Requiring strict positivity at one point of this subset
then gives strict positivity throughout \(\C_+\) by the minimum
principle for harmonic functions.

\begin{lemma}[Identification of analytic laws]\label{lem:analytic-law-identification}
Let \(Q\) and \(Q'\) be probability laws on \(\mathfrak H\).  Suppose that
there is a countable set \(D\subset\C_+\), with an accumulation point in
\(\C_+\), such that the finite-dimensional evaluation laws of \(Q\) and
\(Q'\) agree at every finite subset of \(D\).  Then \(Q=Q'\).
\end{lemma}

\begin{proof}
The evaluation map
\[
 \mathcal E_D:\mathfrak H\longrightarrow\C^D,
 \qquad
 \mathcal E_D(h)=(h(z))_{z\in D},
\]
is Borel, and the identity theorem makes it injective. Since \(\mathfrak H\)
is a standard Borel space, the Lusin--Souslin theorem
\cite[Corollary~15.2, p.~89]{Kechris1995} shows that
\(\mathcal E_D(\mathfrak H)\) is Borel and that \(\mathcal E_D\) has a
measurable inverse on its image. Equality of the finite-dimensional laws on
\(D\) gives \((\mathcal E_D)_\#Q=(\mathcal E_D)_\#Q'\), and hence \(Q=Q'\).
\end{proof}

For a probability law \(Q\) on \(\mathfrak H\), let \(\{Y_j\}_{j\geq1}\)
be a Poisson point process with intensity \(\mathsf{L}_\tau\), and let
\(h_1,h_2,\ldots\) be independent with law \(Q\), independently of the point
process.  Define
\[
 \Phi_{\mathbf Y}(h_1,h_2,\ldots)(z)
 =-\left(z+\sum_{j\geq1}Y_jh_j(z)\right)^{-1}.
\]
For every compact \(K\subset\C_+\), with
\(\eta_K=\inf_{z\in K}\Imag z>0\),
\[
 \sum_{j\geq1}\sup_{z\in K}|Y_jh_j(z)|
 \leq\frac1{\eta_K}\sum_{j\geq1}Y_j<\infty.
\]
The series therefore converges absolutely and locally uniformly. Positivity
of the denominator's imaginary part gives
\(\Phi_{\mathbf Y}(h_1,h_2,\ldots)\in\mathfrak H\). Finite partial sums are
measurable as maps into \(\operatorname{Hol}(\C_+)\), so their locally
uniform limit is measurable. We write \(\mathcal F_\tau(Q)\) for the
resulting law on \(\mathfrak H\).

\begin{theorem}[Critical Poisson cavity recursion]\label{thm:poisson-rde}
Let
\[
 G_\tau(z)
 =\langle\delta_\varnothing,(A_\tau-z)^{-1}
          \delta_\varnothing\rangle,
 \qquad z\in\C_+,
\]
viewed as a random element of \(\mathfrak H\).  Its law \(Q_\tau\) is the
unique probability law on \(\mathfrak H\) satisfying
\[
 Q_\tau=\mathcal F_\tau(Q_\tau).
\]
Equivalently, if \(G_{\tau,j}\) are independent copies of \(G_\tau\),
independent of \(\{Y_j\}\), then
\begin{equation}\label{eq:poisson-rde}
 G_\tau(z)
 \stackrel{\mathrm d}{=}
 -\left(z+\sum_{j\geq1}Y_jG_{\tau,j}(z)\right)^{-1}.
\end{equation}
Moreover,
\[
 \int_\R\frac{\mu_\tau(\mathrm{d}x)}{x-z}
 =\E G_\tau(z).
\]
\end{theorem}

\begin{proof}
Removing the root leaves independent rooted subtrees. The children of each
branch root are governed by a fresh Poisson process with intensity
\(\mathsf L_\tau\), so each branch has the same law as the original tree.
To justify the Schur complement for countably many branches,
we truncate the connections at the root. Set
\[
 b_N=\sum_{j>N}\sqrt{Y_j}\,\delta_{o_j},
\]
where \(o_j\) denotes the root of the \(j\)-th branch, and let
\[
 B_N=|\delta_o\rangle\langle b_N|+|b_N\rangle\langle\delta_o|.
\]
Unlike the finite principal compressions used above, this truncation
removes only connections at the root; each descendant branch remains
unchanged. Set \(A_\tau^{(N)}=A_\tau-B_N\). The root is now coupled only to its first
\(N\) branches, while the others are decoupled from the root. The operator
\(A_\tau^{(N)}\) is self-adjoint because \(B_N\) is bounded and symmetric.
Moreover,
\[
 \operatorname{rank}B_N\leq2,
 \qquad
 \|B_N\|\leq2\|b_N\|
 =2\left(\sum_{j>N}Y_j\right)^{1/2}\longrightarrow0.
\]
Thus \(A_\tau^{(N)}\to A_\tau\) in the norm-resolvent sense. Passing to the
limit in the finite-branch Schur-complement identity shows that \(Q_\tau\)
is a fixed point of \(\mathcal F_\tau\), as expressed in
\eqref{eq:poisson-rde}. The finite principal compressions considered before
Theorem~\ref{thm:poisson-critical} establish the measurability of
\(z\mapsto G_\tau(z)\) as a random element of \(\mathfrak H\).

For uniqueness, let \(Q,Q'\) be two fixed points. The resolvent bound gives
bounded support for their finite-dimensional evaluation laws. The relevant
\(W_1\) distances are therefore finite, and optimal couplings exist. Fix
\(z_1,\ldots,z_r\in\C_+\) with \(\eta_\ell=\Imag z_\ell>1\), and let
\(\nu,\nu'\) be the joint evaluation laws at these points. Equip \(\C^r\)
with the cost
\[
 c(u,v)=\sum_{\ell=1}^r|u_\ell-v_\ell|.
\]
Take an optimal coupling of the joint evaluation vectors with laws
\(\nu,\nu'\), and use independent copies of this vector coupling for the
child messages. Use the same Poisson weights in both recursions, independently of
all message couplings. If \(D_\ell,D'_\ell\) are the resulting denominators,
then
\[
 |D_\ell|,|D'_\ell|\geq\eta_\ell
\]
and
\[
 \E\left|
 -\frac1{D_\ell}+\frac1{D'_\ell}
 \right|
 \leq
 \frac1{\eta_\ell^2}
 \E\sum_{j\geq1}Y_j
 |h_j(z_\ell)-h'_j(z_\ell)|
 =\frac1{\eta_\ell^2}
 \E|h(z_\ell)-h'(z_\ell)|,
\]
because \(\int y\,\mathsf{L}_\tau(\mathrm{d}y)=1\).  Hence the induced map
on joint evaluation laws is a strict Wasserstein contraction, with factor
at most \(\max_\ell\eta_\ell^{-2}<1\).  The finite-dimensional laws of
\(Q\) and \(Q'\) therefore agree throughout the half-plane \(\Imag z>1\).
Choosing a countable dense subset of that half-plane and applying
Lemma~\ref{lem:analytic-law-identification} gives \(Q=Q'\).

Finally, Theorem~\ref{thm:poisson-critical} and the resolvent bound justify
Fubini's theorem, which gives the Stieltjes-transform identity.
\end{proof}

\begin{corollary}[Structure of the critical moments]
\label{cor:critical-moments}
For every \(k\geq1\), \(M_{2k}(\tau)\) is a polynomial in \(\tau^{-1}\)
of degree \(k-1\) with nonnegative integer coefficients.  Its constant term
is \(\Cat_k\), and the coefficient of \(\tau^{1-k}\) is
\(((k-1)!)^2\).  In particular,
\[
 M_2(\tau)=1,
 \qquad
 M_4(\tau)=2+\tau^{-1},
 \qquad
 M_6(\tau)=5+6\tau^{-1}+4\tau^{-2}.
\]
Moreover,
\[
 \mu_\tau\Longrightarrow\mu_{\mathrm{sc}}
 \qquad(\tau\to\infty).
\]
\end{corollary}

\begin{proof}
The power of \(\tau^{-1}\) associated with a class \(\gamma\) is
\(k-\edg(\gamma)\).  The sector \(\edg(\gamma)=k\) consists of the \(\Cat_k\)
Dyck-path classes and contributes the constant term.  The sector
\(\edg(\gamma)=1\) consists of the unique alternating walk on one edge and
contributes \(((k-1)!)^2\tau^{1-k}\).  Direct classification for
\(k=1,2,3\) gives the displayed moments.

To prove the endpoint convergence, let \(\tau_n\to\infty\).  The second
moments of \(\mu_{\tau_n}\) are one, so the family is tight.  For every
fixed \(p\), the Carleman estimate above is uniform for \(\tau_n\geq1\) and
gives
\[
 \sup_n M_{2p+2}(\tau_n)<\infty.
\]
Thus powers of order \(2p\) are uniformly integrable along weakly
convergent subsequences.  Every term with \(\edg(\gamma)<k\) vanishes as
\(\tau_n\to\infty\), whereas the \(\edg(\gamma)=k\) sector contributes
\(\Cat_k\).  Every subsequential limit therefore has the semicircle
moments, and moment determinacy of the semicircle law yields
\(\mu_{\tau_n}\Rightarrow\mu_{\mathrm{sc}}\).
\end{proof}

\subsection{Subcritical collapse and the escaped spectral energy}

When \(d\alpha_d^2\to0\), the law on the \(\sqrt d\) scale converges
weakly to \(\delta_0\), yet its second moment remains one. A vanishing
fraction of the spectral mass, at increasingly large rescaled values,
carries this moment. We first prove this collapse and then identify the
escaping energy by weighting spectral mass with the square of its location.
The following gamma-ratio asymptotics will be used in both steps.
For real \(a>0\) and \(q>0\), write
\[
 (a)_q=\frac{\Gamma(a+q)}{\Gamma(a)}.
\]
This is the real-order extension of the rising factorial introduced above
and agrees with it when \(q\) is a positive integer.

\begin{lemma}[Gamma-ratio asymptotics]\label{lem:gamma-ratio}
For every fixed \(q>0\),
\[
 (a)_q\sim\Gamma(q)a
 \quad(a\downarrow0),
 \qquad
 (a)_q\sim a^q
 \quad(a\to\infty).
\]
\end{lemma}

\begin{proof}
The first relation follows from
\(\Gamma(a+q)\to\Gamma(q)\) and
\(a\Gamma(a)\to1\).  The second is the standard ratio asymptotic
\(\Gamma(a+q)/\Gamma(a)\sim a^q\).
\end{proof}

\begin{theorem}[Subcritical collapse and escape of the second moment]
\label{thm:subcritical-collapse}
Let \(d\to\infty\), let \(\alpha_d>0\), and suppose that
\[
 \tau_d:=d\alpha_d^2\longrightarrow0.
\]
Then
\[
 (\sqrt d\,\cdot)_\#\mu_{d,\alpha_d}
 \Longrightarrow
 \delta_0.
\]
More strongly, for every \(0<p<2\),
\[
 \int_{\R}|x|^p\,
 (\sqrt d\,\cdot)_\#\mu_{d,\alpha_d}(\dd x)
 \longrightarrow0,
\]
whereas
\[
 \int_{\R}x^2\,
 (\sqrt d\,\cdot)_\#\mu_{d,\alpha_d}(\dd x)=1
\]
for every \(d\). Thus \(x^2\) is not uniformly integrable with respect to the
rescaled measures.
\end{theorem}

\begin{proof}
Put \(\theta_d=d\alpha_d\), and fix \(q\in(0,1)\).  At the root of the marked tree,
\[
 S_o:=\langle\delta_o,T_\infty^2\delta_o\rangle
 =\sum_{u\sim o}X_{o,u}X_{u,o}.
\]
Let \(\rho_o^{(2)}\) be the spectral measure of \(T_\infty^2\) at
\(\delta_o\).  By functional calculus,
\[
 \langle\delta_o,(T_\infty^2)^q\delta_o\rangle
 =\int_{[0,1]}x^q\,\rho_o^{(2)}(\dd x),
 \qquad
 S_o=\int_{[0,1]}x\,\rho_o^{(2)}(\dd x).
\]
Jensen's inequality, applied to the probability measure \(\rho_o^{(2)}\)
and the concave function \(x\mapsto x^q\), gives
\[
 \langle\delta_o,(T_\infty^2)^q\delta_o\rangle
 \leq S_o^q.
\]
Since \(q<1\), subadditivity and independence of marks at distinct vertices
yield
\begin{align*}
 \E S_o^q
 &\leq
 \sum_{u\sim o}\E[X_{o,u}^qX_{u,o}^q]\\
 &=d\left(\frac{(\alpha_d)_q}{(\theta_d)_q}\right)^2.
\end{align*}
Consequently,
\[
 \int |x|^{2q}\,
 (\sqrt d\,\cdot)_\#\mu_{d,\alpha_d}(\dd x)
 \leq
 d^{q+1}
 \left(\frac{(\alpha_d)_q}{(\theta_d)_q}\right)^2.
\]
We show that the right-hand side tends to zero by considering subsequences.
Every subsequence has a further subsequence on which
\(\theta_d\to\theta\in[0,\infty]\). Since \(\tau_d\to0\), we also have
\(\alpha_d\to0\). Lemma~\ref{lem:gamma-ratio} gives three cases. If
\(\theta=0\), then
\((\alpha_d)_q/(\theta_d)_q\sim\alpha_d/\theta_d=1/d\), so the bound is
asymptotic to \(d^{q-1}\). If \(\theta\in(0,\infty)\), the bound is
\(O(d^{q+1}\alpha_d^2)=O(\theta_d^2d^{q-1})\). If \(\theta=\infty\), then
\[
 d^{q+1}
 \left(\frac{(\alpha_d)_q}{(\theta_d)_q}\right)^2
 \sim
 \Gamma(q)^2\tau_d^{1-q}.
\]
All three bounds tend to zero. Because every subsequence has a further
subsequence with this property, the original bound also tends to zero.
Taking \(p=2q\) gives vanishing fractional moments of every positive order
below two. Markov's inequality then gives weak convergence to \(\delta_0\).
The second-moment identity follows from \(m_2(d,\alpha)=1/d\).
\end{proof}

Combining the preceding theorem with the semicircle and critical theorems
gives a complete classification on the variance scale. Write
\[
 \widehat\mu_{d,\alpha}=(\sqrt d\,\cdot)_\#\mu_{d,\alpha},
 \qquad \overline{\R}_+=[0,\infty],
\]
where \(\overline{\R}_+\) has the compact topology induced by
\(\tau\mapsto\tau/(1+\tau)\), with \(\infty\) mapped to \(1\), and put
\[
 \Phi(\tau)=
 \begin{cases}
  \delta_0,&\tau=0,\\
  \mu_\tau,&0<\tau<\infty,\\
  \mu_{\mathrm{sc}},&\tau=\infty.
 \end{cases}
\]

\begin{corollary}[Complete variance-scale classification]
\label{cor:sqrt-d-phase-diagram}
Let \(d\to\infty\), \(\alpha_d>0\), and \(\tau_d=d\alpha_d^2\).
The weak subsequential limits of \(\widehat\mu_{d,\alpha_d}\) are exactly the
laws \(\Phi(\tau)\) arising from subsequential limits
\(\tau\in\overline{\R}_+\) of \(\tau_d\). Consequently,
\(\widehat\mu_{d,\alpha_d}\) has a weak limit if and only if \(\tau_d\)
converges in \(\overline{\R}_+\). If \(\tau_d\to\tau\), then
\[
 W_1\bigl(\widehat\mu_{d,\alpha_d},\Phi(\tau)\bigr)\longrightarrow0.
\]
In particular,
\[
 \widehat\mu_{d,\alpha_d}\Longrightarrow\mu_{\mathrm{sc}}
 \quad\Longleftrightarrow\quad d\alpha_d^2\longrightarrow\infty.
\]
At \(\tau=0\), the rescaled laws retain second moment one, but this moment
escapes the weak limit.
\end{corollary}

\begin{proof}
Along a sequence with \(\tau_d\to\tau\),
Theorems~\ref{thm:semicircle}, \ref{thm:critical-family}, and
\ref{thm:subcritical-collapse} give the weak limit \(\Phi(\tau)\).
The identity \(\int x^2\,\widehat\mu_{d,\alpha_d}(\dd x)=1\) implies
\[
 \int_{|x|>R}|x|\,\widehat\mu_{d,\alpha_d}(\dd x)\leq R^{-1},\qquad R>0,
\]
which gives uniform integrability of first moments and hence \(W_1\)
convergence. From any weakly convergent subsequence, compactness of
\(\overline{\R}_+\) yields a further subsequence on which \(\tau_d\)
converges. The same three theorems identify its limiting law. Conversely,
each parameter subsequential limit yields the corresponding law.

The laws \(\Phi(\tau)\) are pairwise distinct. At \(\tau=0\), the second
moment is zero; for every other law it is one. For \(0<\tau<\infty\), the
fourth moment is \(2+\tau^{-1}\), whereas the semicircle law has fourth moment
\(2\), by Corollary~\ref{cor:critical-moments}. Thus, if the laws converge,
all subsequential limits of the parameter must coincide. A sequence in the
compact space \(\overline{\R}_+\) with a unique subsequential limit converges.
This proves both the converse and the semicircle characterization.
\end{proof}

We next transfer the classification to the quantum walk. The comparison
concerns the normalized inherited component, with its two phase branches
centered at \(\pi/2\) and \(-\pi/2\), rather than the full unitary law.

\begin{corollary}[Quantitative inherited eigenphase transition]
\label{cor:eigenphase-diagram}
Let \(\vartheta(x)=\arcsin x\) for \(-1\leq x\leq1\). The inherited
pair is \(\phi_\pm(x)=\exp\{\pm\mathrm{i}(\pi/2-\vartheta(x))\}\).
Centering the two phases at \(\pi/2\) and \(-\pi/2\) gives the signed
coordinates \(-\vartheta(x)\) and \(\vartheta(x)\). By symmetry of
\(\mu_{d,\alpha}\), their normalized mixture is
\(\vartheta_\#\mu_{d,\alpha}\). Set
\[
 \widehat\eta_{d,\alpha}
 = (\sqrt d\,\vartheta)_\#\mu_{d,\alpha},
 \qquad C_{\mathrm{ph}}=\frac\pi2-1.
\]
For every \(d\geq3\) and \(\alpha>0\),
\begin{equation}\label{eq:eigenphase-W1-bound}
 W_1(\widehat\eta_{d,\alpha},\widehat\mu_{d,\alpha})
 \leq C_{\mathrm{ph}}\sqrt{m_4(d,\alpha)}
 \leq \frac{C_{\mathrm{ph}}}{\sqrt d}.
\end{equation}
For every sequence \(d\to\infty\) and \(\alpha_d>0\), the centered eigenphase
laws have a weak limit if and only if \(d\alpha_d^2\) converges in
\(\overline{\R}_+\). Their weak subsequential limits are those described in
Corollary~\ref{cor:sqrt-d-phase-diagram}. If \(d\alpha_d^2\to\tau\), then
\(\widehat\eta_{d,\alpha_d}\to\Phi(\tau)\) in \(W_1\). The comparison in
\eqref{eq:eigenphase-W1-bound} is \(O(d^{-1})\) whenever
\(\sup_d d^2m_4(d,\alpha_d)<\infty\). This includes the semicircular regime
and critical regimes with a finite positive parameter. The full unitary law
converges to \((\delta_1+\delta_{-1})/2\).
\end{corollary}

\begin{proof}
The expansion
\[
 \arcsin t-t=\sum_{j\geq1}a_jt^{2j+1},
 \qquad a_j=\frac{\binom{2j}{j}}{4^j(2j+1)}>0,
 \qquad \sum_{j\geq1}a_j=\frac\pi2-1,
\]
valid on \([0,1]\), yields \(|\arcsin x-x|\leq C_{\mathrm{ph}}|x|^3\) for
\(|x|\leq1\). Couple the two laws using the same \(X\sim\mu_{d,\alpha}\).
We have \(\E X^2=m_2(d,\alpha)=1/d\). Since \(|X|\leq1\), we also have
\(m_4(d,\alpha)\leq m_2(d,\alpha)=1/d\). Cauchy--Schwarz gives
\(\E|X|^3\leq(\E X^2\,\E X^4)^{1/2}\), so
\[
 W_1(\widehat\eta_{d,\alpha},\widehat\mu_{d,\alpha})
 \leq C_{\mathrm{ph}}\sqrt d\,\E|X|^3
 \leq C_{\mathrm{ph}}\sqrt{m_4(d,\alpha)}
 \leq \frac{C_{\mathrm{ph}}}{\sqrt d}.
\]
The two families therefore have the same weak subsequential limits.
Corollary~\ref{cor:sqrt-d-phase-diagram} gives the classification and the
\(W_1\) limits. The fourth-moment convergence in
Theorems~\ref{thm:semicircle} and \ref{thm:critical-family} gives the
\(O(d^{-1})\) comparison. The claim for the full unitary law follows from
the total-variation estimate at the start of this section.
\end{proof}

We next describe the spectral energy that escapes the weak-limit scale. Set
\[
 \theta_d=d\alpha_d,
 \qquad
 s_d=1+\theta_d.
\]
Because \(d m_2(d,\alpha_d)=1\), the measure
\[
 \pi_d^{\mathrm{en}}
 :=
 (s_d\,\cdot)_\#
 \bigl[d\,x^2\mu_{d,\alpha_d}(\dd x)\bigr]
\]
is a probability measure. The factor \(d x^2\) weights spectral mass by its contribution to the
second moment; the pushforward under \(x\mapsto s_d x\) then changes
the observation scale. We call the resulting measure the
spectral-energy-biased law.

\begin{theorem}[Subcritical spectral-energy profile]
\label{thm:subcritical-energy}
Assume that
\[
 d\alpha_d^2\longrightarrow0,
 \qquad
 d\alpha_d\longrightarrow\theta\in\overline{\R}_+.
\]
Let \(\chi\) be a Rademacher variable independent of all variables below. For
\(0<\theta<\infty\), let \(B_\theta,B_\theta'\) be independent
\(\operatorname{Beta}(1,\theta)\) variables. Set \(B_0=B_0'=1\) almost surely,
consistent with the weak limit as \(\theta\downarrow0\). At the endpoint
\(\theta=\infty\), let \(E,E'\) be independent mean-one exponential variables.
Define
\[
 Z_\theta=
 \begin{cases}
  \chi(1+\theta)\sqrt{B_\theta B_\theta'},
       &0\leq\theta<\infty,\\
  \chi\sqrt{EE'},&\theta=\infty.
 \end{cases}
\]
Then
\[
 \pi_d^{\mathrm{en}}
 \Longrightarrow
 \pi_\theta:=\mathcal{L}(Z_\theta),
\]
and all moments converge. The odd moments vanish, and
\[
 \int x^{2k}\,\pi_\theta(\dd x)
 =
 \begin{cases}
  (1+\theta)^{2k}
  \left[\dfrac{k!}{(1+\theta)_k}\right]^2,
       &0\leq\theta<\infty,\\[1.2ex]
  (k!)^2,&\theta=\infty.
 \end{cases}
\]
In particular, \(\pi_0=\tfrac12(\delta_{-1}+\delta_1)\), while
\(\pi_\infty\) has unbounded support.
\end{theorem}

\begin{proof}
The total mass of \(\pi_d^{\mathrm{en}}\) is
\(d m_2(d,\alpha_d)=1\).  Since \(\mu_{d,\alpha_d}\) is symmetric,
\(\pi_d^{\mathrm{en}}\) is symmetric as well, so all odd moments vanish.
For \(k\geq0\),
\[
 \int x^{2k}\,\pi_d^{\mathrm{en}}(\dd x)
 =d s_d^{2k}m_{2k+2}(d,\alpha_d).
\]
Set \(K=k+1\) and use the grouped expansion from
Proposition~\ref{prop:grouped-moments}.  The unique class with one traced
edge contributes
\begin{align*}
 L_{d,k}
 &=d^2s_d^{2k}
   \frac{(\alpha_d)_{k+1}^2}{(\theta_d)_{k+1}^2}\\
 &=s_d^{2k}
   \left[\frac{(\alpha_d)_{k+1}/\alpha_d}
   {(1+\theta_d)_k}\right]^2.
\end{align*}
Since \(\alpha_d\to0\),
\[
 \frac{(\alpha_d)_{k+1}}{\alpha_d}\longrightarrow k!.
\]
If \(\theta<\infty\), this gives
\[
 L_{d,k}\longrightarrow
 (1+\theta)^{2k}
 \left[\frac{k!}{(1+\theta)_k}\right]^2.
\]
If \(\theta=\infty\), then
\(s_d^k/(1+\theta_d)_k\to1\), and the limit is \((k!)^2\).

We now bound the contribution of every class with at least two traced edges.
Fix \(K\), and take \(d\) large enough that \(d\geq2K\) and
\(\alpha_d\leq1\). For \(1\leq r\leq K\),
\[
 (\alpha_d)_r\leq\alpha_d r!.
\]
Moreover, for every positive integer \(R\),
\[
 (\theta_d)_R
 \geq
 \theta_d(1+\theta_d)^{R-1}
 =\theta_d s_d^{R-1}.
\]
A canonical class with \(e\geq2\) traced edges has \(e+1\) vertices. Using
\(N_d(\gamma)\leq d^e\) and \(\sum_vR_v=2K\), we bound its contribution to
\(d s_d^{2K-2}m_{2K}(d,\alpha_d)\) by
\begin{align*}
 C_K
 \frac{d^{e+1}s_d^{2K-2}\alpha_d^{2e}}
 {\theta_d^{e+1}s_d^{2K-e-1}}
 &=C_K[\alpha_d s_d]^{e-1}\\
 &=C_K(\alpha_d+d\alpha_d^2)^{e-1}.
\end{align*}
Here
\[
 \alpha_ds_d=\alpha_d+d\alpha_d^2\longrightarrow0.
\]
The bound therefore tends to zero. Since \(\mathcal W_K\) is finite,
summing over the classes proves convergence of every even moment.

For finite \(\theta\),
\[
 \E[B_\theta^k]=\frac{k!}{(1+\theta)_k},
\]
with the same formula under the convention \(B_0=1\). At the infinite
endpoint, \(\E[E^k]=k!\). The limits obtained above are therefore the moments
of \(Z_\theta\). Convergence of the second moments to one gives tightness.
For each fixed \(k\), convergence of the \(2k+2\)-th moments gives an
eventual uniform bound. Thus \(|x|^{2k}\) is uniformly integrable along
weakly convergent subsequences. For finite \(\theta\), the target has compact
support and is moment-determinate. At \(\theta=\infty\), the even moments
are \((k!)^2\), and
\[
 ((k!)^2)^{-1/(2k)}=(k!)^{-1/k}\sim\frac{\mathrm e}{k},
\]
so Carleman's condition holds. Every weak subsequential limit has the moments
of \(Z_\theta\). Moment determinacy then gives weak convergence, together
with convergence of all moments.
\end{proof}

\begin{corollary}[Classification of the escaped-energy limits]
\label{cor:energy-classification}
The map \(\theta\mapsto\pi_\theta\) from \(\overline{\R}_+\) into
probability measures with the \(W_1\) metric is continuous and injective.
Let \(d\to\infty\) and \(d\alpha_d^2\to0\), without requiring
\(\theta_d=d\alpha_d\) to converge. The weak subsequential limits of
\(\pi_d^{\mathrm{en}}\) are exactly the laws \(\pi_\theta\) arising from
subsequential limits \(\theta\in\overline{\R}_+\) of \(\theta_d\). Hence
\(\pi_d^{\mathrm{en}}\) has a weak limit if and only if \(\theta_d\)
converges in \(\overline{\R}_+\). In that case, convergence also holds in
\(W_1\).
\end{corollary}

\begin{proof}
For each \(k\geq0\), the even-moment formula in
Theorem~\ref{thm:subcritical-energy} is continuous in
\(\theta\in\overline{\R}_+\) and bounded by \((k!)^2\), since
\((1+\theta)^k\leq(1+\theta)_k\) for finite \(\theta\geq0\),
with equality in the bound at \(\theta=\infty\); the odd moments
vanish. The target laws \(\pi_\theta\) all have second moment one, which gives
tightness and uniform integrability of first moments. Higher-moment bounds
allow moments to pass to weak subsequential limits. Moment determinacy from
the same theorem therefore gives \(W_1\) continuity. To prove injectivity,
observe that
\[
 \int x^4\,\pi_\theta(\dd x)
 =4\left(\frac{1+\theta}{2+\theta}\right)^2
 \quad(0\leq\theta<\infty),
 \qquad \int x^4\,\pi_\infty(\dd x)=4.
\]
The fourth moment is strictly increasing on \(\overline{\R}_+\). For the
approximating laws, Theorem~\ref{thm:moment-formula} gives the exact second moment
\begin{align*}
 \int y^2\,\pi_d^{\mathrm{en}}(\dd y)
 &=d(1+d\alpha_d)^2m_4(d,\alpha_d)\\
 &=1+\left(4-\frac2d\right)\alpha_d
       +(2d-1)\alpha_d^2\longrightarrow1,
\end{align*}
since \(d\alpha_d^2\to0\) implies \(\alpha_d\to0\). The approximating laws
are therefore tight, with uniformly integrable first moments. Pass to further
subsequences along which the parameter converges in the compact space
\(\overline{\R}_+\), and apply Theorem~\ref{thm:subcritical-energy}. This
identifies all weak subsequential limits. Injectivity proves the converse,
and uniform integrability gives \(W_1\) convergence whenever a weak limit
exists.
\end{proof}

The energy-biased law can also be converted back to a spectral intensity
by dividing out the quadratic weight. This measure is naturally considered
away from the origin, on \(\R^*=\R\setminus\{0\}\). Define
\[
 \Lambda_d
 =\frac{d}{s_d^2}(s_d\,\cdot)_\#\mu_{d,\alpha_d}.
\]
Then \(\pi_d^{\mathrm{en}}(\dd x)=x^2\Lambda_d(\dd x)\).

\begin{corollary}[Subcritical spectral-intensity limit]
\label{cor:subcritical-intensity}
Under the assumptions of Theorem~\ref{thm:subcritical-energy},
\[
 \Lambda_d
 \xrightarrow{\mathrm v}
 \Lambda_\theta,
 \qquad
 \Lambda_\theta(\dd x)=x^{-2}\pi_\theta(\dd x),
\]
vaguely on \(\R^*\).  The limit is locally finite on \(\R^*\) and satisfies
\[
 \int_{\R^*}x^2\,\Lambda_\theta(\dd x)=1.
\]
At the two endpoints,
\[
 \Lambda_0=\frac12(\delta_{-1}+\delta_1),
\]
whereas
\[
 \pi_\infty(\dd x)
 =2|x|K_0(2|x|)\,\dd x,
 \qquad
 \Lambda_\infty(\dd x)
 =\frac{2K_0(2|x|)}{|x|}\,\dd x,
 \quad x\neq0,
\]
where \(K_0\) is the modified Bessel function of the second kind.
\end{corollary}

\begin{proof}
For \(f\in C_c(\R^*)\), the support of \(f\) has positive distance
from the origin.  Hence
\[
 g_f(x)=
 \begin{cases}
  f(x)/x^2,&x\neq0,\\
  0,&x=0
 \end{cases}
\]
is a bounded continuous function on \(\R\), and
\[
 \int f(x)\,\Lambda_d(\dd x)
 =\int g_f(x)\,\pi_d^{\mathrm{en}}(\dd x).
\]
Theorem~\ref{thm:subcritical-energy} gives the vague limit. The second-moment
identity follows because \(\pi_\theta\) has total mass one and no atom at
zero. Indeed, the Beta and exponential variables are strictly positive almost
surely, and \(\pi_0\) is supported on \(\{\pm1\}\). For \(\theta=\infty\),
the product \(Y=EE'\) has density
\[
 y\longmapsto2K_0(2\sqrt y),
 \qquad y>0,
\]
by the integral representation of \(K_0\)
\cite[Eq.~10.32.10]{NISTDLMF}.  Passing from \(Y\) to
\(\chi\sqrt Y\) gives the displayed densities.
\end{proof}

\begin{remark}[The Bessel intensity in the subcritical profile]
The intensity measure appearing in
Corollary~\ref{cor:subcritical-intensity} at \(\theta=\infty\),
\[
 \frac{2K_0(2|x|)}{|x|}\,\mathrm{d}x,
\]
is the symmetric square-root pushforward of \(\mathsf{L}_1\).
Thus the same Bessel measure describes the critical tree's squared edge
weights and, after this transformation, the escaped-energy intensity.
\end{remark}

The two families meet at different parameter endpoints: as the critical
parameter \(\tau\) decreases to zero, a further energy weighting and
rescaling recover the \(\theta=\infty\) member of the subcritical family.
The next result makes this connection precise.

\begin{corollary}[The lower endpoint of the critical family]
\label{cor:critical-zero-endpoint}
For every \(q\in(0,1)\),
\[
 \int_{\R}|x|^{2q}\,\mu_\tau(\dd x)
 \leq
 \Gamma(q)^2\tau^{1-q},
 \qquad \tau>0.
\]
Consequently,
\[
 \mu_\tau\Longrightarrow\delta_0
 \qquad(\tau\downarrow0).
\]
Moreover, if
\[
 \Pi_\tau
 :=(\sqrt\tau\,\cdot)_\#
   [x^2\mu_\tau(\dd x)],
\]
then
\[
 \Pi_\tau\Longrightarrow\pi_\infty
 \qquad(\tau\downarrow0),
\]
and equivalently
\[
 \tau^{-1}(\sqrt\tau\,\cdot)_\#\mu_\tau
 \xrightarrow{\mathrm v}
 \Lambda_\infty
\]
vaguely on \(\R^*\).
\end{corollary}

\begin{proof}
For fixed \(\tau>0\), take any sequence with
\(d\alpha_d^2\to\tau\).  Then \(d\alpha_d\to\infty\), and the fractional
moment estimate in the proof of Theorem~\ref{thm:subcritical-collapse}
gives
\[
 \limsup_{d\to\infty}
 \int |x|^{2q}\,
 (\sqrt d\,\cdot)_\#\mu_{d,\alpha_d}(\dd x)
 \leq\Gamma(q)^2\tau^{1-q}.
\]
Theorem~\ref{thm:critical-family} and the Portmanteau theorem, applied to
the nonnegative lower semicontinuous function \(|x|^{2q}\), yield the same
bound for \(\mu_\tau\).  For every \(\varepsilon>0\), Markov's inequality gives
\[
 \mu_\tau(|x|>\varepsilon)
 \leq\varepsilon^{-2q}\Gamma(q)^2\tau^{1-q}
 \longrightarrow0,
\]
which proves weak convergence to \(\delta_0\).

For the energy-biased law, note first that \(\Pi_\tau\) is a probability
measure because \(M_2(\tau)=1\).  Moreover,
\[
 \int x^{2k}\,\Pi_\tau(\dd x)
 =\tau^kM_{2k+2}(\tau).
\]
In the closed-walk formula for \(M_{2k+2}(\tau)\), the unique one-edge class
contributes \((k!)^2\tau^{-k}\). After multiplication by \(\tau^k\), every
class with at least two edges vanishes in the limit. The moments therefore
converge to \((k!)^2\), those of \(\pi_\infty\). Moreover,
\[
 \int x^2\,\Pi_\tau(\dd x)
 =\tau M_4(\tau)=1+2\tau,
\]
so the family is tight as \(\tau\downarrow0\).  For each fixed \(k\),
the next even moment has the explicit nonnegative-coefficient expansion
\[
 \tau^{k+1}M_{2k+4}(\tau)
 =\sum_{\gamma\in\mathcal W_{k+2}}
 c_\gamma\,\tau^{\edg(\gamma)-1},
 \qquad
 c_\gamma=\prod_{e\in E(H_\gamma)}
 \bigl((r_e(\gamma)-1)!\bigr)^2\geq0,
\]
and is bounded on \((0,1]\). This bound gives the uniform integrability
needed to pass moments to weak subsequential limits. Carleman's condition
for \(\pi_\infty\) then implies \(\Pi_\tau\Rightarrow\pi_\infty\).

Finally, for \(f\in C_c(\R^*)\), apply this weak convergence to the bounded
continuous extension of \(f(x)/x^2\).  The identity
\[
 \Pi_\tau(\dd x)
 =x^2\tau^{-1}(\sqrt\tau\,\cdot)_\#\mu_\tau(\dd x)
\]
then gives the asserted vague convergence.
\end{proof}

\subsection{Uniform approximation across the phase diagram}
\label{subsec:uniform-approximation}

The three regimes combine into a single approximation, uniform over all
positive Dirichlet parameters. The approximating law is indexed by the
current value of \(d\alpha^2\), so no convergence of this parameter needs
to be assumed.

\begin{proposition}[Uniform Wasserstein approximation]
\label{prop:uniform-phase-approximation}
The map \(\Phi:\overline{\R}_+\to\mathcal P_1(\R)\) is continuous and
injective, where \(\mathcal P_1(\R)\) denotes the probability measures with
finite first moment, equipped with \(W_1\). Moreover,
\begin{align}
 \lim_{d\to\infty}\sup_{\alpha>0}
 W_1\bigl(\widehat\mu_{d,\alpha},\Phi(d\alpha^2)\bigr)&=0,
 \label{eq:uniform-discriminant-approximation}\\
 \lim_{d\to\infty}\sup_{\alpha>0}
 W_1\bigl(\widehat\eta_{d,\alpha},\Phi(d\alpha^2)\bigr)&=0.
 \label{eq:uniform-eigenphase-approximation}
\end{align}
\end{proposition}

\begin{proof}
Suppose \(\tau_j\to\tau\in(0,\infty)\). Each critical moment is continuous
in \(\tau\), and its polynomial formula in \(\tau^{-1}\), with nonnegative
coefficients, gives local uniform bounds. The common second moment gives tightness; the next
even-moment bound allows each moment to pass to a weak subsequential limit.
Moment determinacy in Theorem~\ref{thm:critical-family} therefore gives
\(\mu_{\tau_j}\Rightarrow\mu_\tau\). Uniform integrability of first moments
then gives \(W_1\) convergence.
At infinity, Corollary~\ref{cor:critical-moments} and the same
second-moment bound give \(W_1(\mu_\tau,\mu_{\mathrm{sc}})\to0\).
At zero, Corollary~\ref{cor:critical-zero-endpoint}, with \(q=1/2\),
gives the explicit bound
\[
 W_1(\mu_\tau,\delta_0)
 =\int |x|\,\mu_\tau(\dd x)\leq\pi\sqrt\tau.
\]
This proves continuity on \(\overline{\R}_+\). Injectivity was proved in
Corollary~\ref{cor:sqrt-d-phase-diagram}.

Suppose \eqref{eq:uniform-discriminant-approximation} fails. Then there exist
\(\varepsilon>0\), integers \(d_j\to\infty\), and \(\alpha_j>0\) such that
the corresponding distance is at least \(\varepsilon\). Compactness yields a
further subsequence with \(d_j\alpha_j^2\to\tau\in\overline{\R}_+\).
Corollary~\ref{cor:sqrt-d-phase-diagram} gives
\(\widehat\mu_{d_j,\alpha_j}\to\Phi(\tau)\) in \(W_1\). Continuity gives
\(\Phi(d_j\alpha_j^2)\to\Phi(\tau)\) in \(W_1\). The triangle inequality
contradicts the lower bound. Finally, the triangle inequality and
\eqref{eq:eigenphase-W1-bound} imply
\eqref{eq:uniform-eigenphase-approximation}.
\end{proof}

\section{Conclusions and further directions}\label{sec:conclusion}

We have studied a sparse self-adjoint random operator in which independent
Dirichlet vectors on vertex stars couple the incident weights. The operator
is the discriminant of a quenched reflection quantum walk, giving an exact
spectral relation between the two. At fixed degree, empirical spectral
measures along locally tree-like graph sequences converge almost surely to
a deterministic tree law with an all-order closed-walk formula. Rare paths
give spectrum and essential spectrum \([-1,1]\) almost surely, while both
amplitude-disorder endpoint laws have strictly smaller support.

For expected root spectral laws on the infinite tree, variance-scale
convergence is equivalent to convergence of \(d\alpha_d^2\) in
\([0,\infty]\). Finite positive parameter limits give the Bessel--Poisson
critical family; the two endpoints give subcritical collapse and the semicircle
law. Including the endpoints yields a \(W_1\) approximation uniform over
\(\alpha>0\). Each critical law is the expected root spectral measure of an
almost surely essentially self-adjoint Poisson-tree adjacency operator with
an explicit Bessel L\'{e}vy intensity. Below criticality, the Beta-product
energy family recovers the escaping second moment. Its convergence is
classified by the second parameter \(d\alpha_d\). A uniform quantitative
comparison transfers the variance-scale classification and approximation to
centered inherited eigenphases. The full unitary law remains dominated by
the structural atoms at \(\pm1\).

These limits lead to further questions about the underlying operators.
A simultaneous finite-graph limit with \(n,d\to\infty\) calls for
degree-uniform local approximation and moment estimates. The representation in
Theorem~\ref{thm:poisson-critical} also motivates studying convergence of the
rescaled edge-weight point processes and their associated operators or cavity
maps. At finite \(\alpha\), the explicit fractional moments and the
weighted-message recursion provide starting points for studying spectral type,
localization, and delocalization.

\appendix
\section{Finite-degree cavity recursion and its analytic fixed point}
\label{sec:cavity}

Removing an edge from a tree separates its two branches. This gives a
recursion for the Green function, but the Dirichlet constraint creates a
dependence that must remain in the recursion. We record it by weighting the
branch Green function with the mass assigned to the deleted edge.

We use the Herglotz resolvent convention
\[
 R_H(z)=(H-z)^{-1},
 \qquad z\in\C_+.
\]
With this convention, diagonal resolvent entries have positive imaginary
part. The finite-degree recursion below is the bounded-operator counterpart
of the critical Poisson recursion in Theorem~\ref{thm:poisson-rde}.

We work in the positive gauge on the tree, so that
\[
 a_{xy}=\sqrt{X_{x,y}X_{y,x}}.
\]
For an oriented edge \(x\to p\), remove \(\{x,p\}\) and let
\(\mathbb{T}_{x\to p}\) be the component containing \(x\). Denote by
\(T_{x\to p}\) the compression of \(T_\infty\) to
\(\ell^2(\mathbb{T}_{x\to p})\), and define
\[
 \Gamma_{x\to p}(z)
 =
 \left\langle
 \delta_x,(T_{x\to p}-z)^{-1}\delta_x
 \right\rangle.
\]
The coordinate \(X_{x,p}\) remains part of the Dirichlet mark at \(x\),
even though the edge to \(p\) has been deleted from the cavity operator.

\begin{lemma}[Herglotz convention and resolvent bounds]
For every oriented edge \(x\to p\) and every \(z\in\C_+\),
\[
 \Gamma_{x\to p}(z)\in\C_+,
 \qquad
 |\Gamma_{x\to p}(z)|\leq\frac1{\Imag z}.
\]
Moreover,
\[
 \Imag\Gamma_{x\to p}(z)
 =
 (\Imag z)
 \left\|(T_{x\to p}-z)^{-1}\delta_x\right\|^2,
\]
and
\[
 -\mathrm{i}y\,\Gamma_{x\to p}(\mathrm{i}y)
 \longrightarrow1
 \qquad(y\to\infty).
\]
\end{lemma}

\begin{proof}
The resolvent identity and self-adjointness give
\[
 \frac{(T_{x\to p}-z)^{-1}-(T_{x\to p}-\overline z)^{-1}}{2\mathrm{i}}
 =
 (\Imag z)
 (T_{x\to p}-z)^{-1}(T_{x\to p}-\overline z)^{-1}.
\]
Taking the matrix element at \(\delta_x\) proves the formula for the
imaginary part.  The norm bound follows from
\(\|(T_{x\to p}-z)^{-1}\|\leq(\Imag z)^{-1}\).  Finally,
\[
 (T_{x\to p}-\mathrm{i}y)^{-1}
 =
 -\frac1{\mathrm{i}y}
 \left(I-\frac{T_{x\to p}}{\mathrm{i}y}\right)^{-1},
\]
and \(\|T_{x\to p}\|\leq1\), which yields the asymptotic relation.
\end{proof}

The masses assigned to the child directions sum to \(1-X_{x,p}\), so
\(\Gamma_{x\to p}\) depends on the mass assigned to the deleted parent
edge. We retain this dependence by defining the weighted cavity message
\[
 \mathfrak m_{x\to p}(z)
 =X_{x,p}\Gamma_{x\to p}(z).
\]
Since \(X_{x,p}>0\) almost surely for \(\alpha>0\), this message satisfies
\[
 \mathfrak m_{x\to p}(z)\in\C_+,
 \qquad
 |\mathfrak m_{x\to p}(z)|\leq\frac1{\Imag z},
\]
and
\[
 -\mathrm{i}y\,\mathfrak m_{x\to p}(\mathrm{i}y)
 \longrightarrow X_{x,p}
 \qquad(y\to\infty).
\]

\begin{proposition}[Closed weighted-message recursion]\label{prop:cavity-rde}
Let \(\mathfrak m\) have the law of the canonical weighted message
\(\mathfrak m_{x\to p}\), and let
\[
 (X_0,X_1,\ldots,X_{d-1})
 \sim\operatorname{Dirichlet}(\alpha,\ldots,\alpha),
\]
where \(X_0\) is the mass assigned at $x$ to the parent direction and \(X_1,\ldots,X_{d-1}\) are the masses assigned to the child directions. Let \(\mathfrak m_1,\ldots,\mathfrak m_{d-1}\) be independent copies
of \(\mathfrak m\), independent of the Dirichlet vector. The following equality is an equality in law of random analytic
functions (and hence of all their finite-dimensional evaluations):
\[
 \mathfrak m(z)
 \stackrel{\mathrm d}{=}
 -\frac{X_0}
 {z+\displaystyle\sum_{j=1}^{d-1}X_j\mathfrak m_j(z)},
 \qquad z\in\C_+.
\]
For the root formula, take \(d\) independent copies
\(\mathfrak m_1,\ldots,\mathfrak m_d\) of \(\mathfrak m\) and an independent
vector \((X_1,\ldots,X_d)\sim\operatorname{Dirichlet}(\alpha,\ldots,\alpha)\).
Then
\[
 G_o(z)
 :=
 \left\langle\delta_o,(T_\infty-z)^{-1}\delta_o\right\rangle
 \stackrel{\mathrm d}{=}
 -\frac1
 {z+\displaystyle\sum_{j=1}^{d}X_j\mathfrak m_j(z)}.
\]
Consequently,
\[
 s_{d,\alpha}(z)
 :=
 \int_{\R}\frac{\mu_{d,\alpha}(\dd x)}{x-z}
 =\E G_o(z).
\]
\end{proposition}

\begin{proof}
Let \(y_1,\ldots,y_{d-1}\) be the children of \(x\) in
\(\mathbb T_{x\to p}\).  With respect to the orthogonal decomposition
\[
 \ell^2(\mathbb T_{x\to p})
 =
 \mathbb C\delta_x
 \oplus
 \bigoplus_{j=1}^{d-1}\ell^2(\mathbb T_{y_j\to x}),
\]
the operator \(T_{x\to p}-z\) has block form
\[
 \begin{pmatrix}
  -z & B^*\\
  B & \displaystyle\bigoplus_{j=1}^{d-1}(T_{y_j\to x}-z)
 \end{pmatrix},
\]
where
\[
 B\delta_x
 =\bigoplus_{j=1}^{d-1}a_{x,y_j}\delta_{y_j}.
\]
The lower-right block is invertible for \(z\in\mathbb C_+\), and the Schur
complement gives
\[
 \Gamma_{x\to p}(z)
 =
 -\left[
 z+\sum_{j=1}^{d-1}
 a_{x,y_j}^2\Gamma_{y_j\to x}(z)
 \right]^{-1}.
\]
Since
\[
 a_{x,y_j}^2\Gamma_{y_j\to x}(z)
 =X_{x,y_j}\mathfrak m_{y_j\to x}(z),
\]
multiplying by \(X_{x,p}\) gives the first recursion. The marked child
subtrees are independent and identically distributed, and independent of the
mark at \(x\). The same Schur complement with \(d\) branches gives the root
formula. If \(\rho_o\) denotes the root spectral measure, then
\[
 G_o(z)=\int_\R\frac{\rho_o(\dd x)}{x-z},
 \qquad \mu_{d,\alpha}=\E\rho_o.
\]
The resolvent bound justifies interchanging expectation and the spectral
integral, giving the final identity.
\end{proof}

\begin{remark}[Dirichlet neutrality]
Since \(X_0<1\) almost surely, we may write
\[
 Y=X_0,
 \qquad
 Z_j=\frac{X_j}{1-X_0},
 \quad 1\leq j\leq d-1.
\]
Then
\[
 Y\sim\operatorname{Beta}(\alpha,(d-1)\alpha),
 \qquad
 (Z_1,\ldots,Z_{d-1})
 \sim\operatorname{Dirichlet}(\alpha,\ldots,\alpha),
\]
and \(Y\) is independent of \(Z\).  Hence the recursion can equivalently be
written as
\[
 \mathfrak m(z)
 \stackrel{\mathrm d}{=}
 -\frac{Y}
 {z+(1-Y)\displaystyle\sum_{j=1}^{d-1}Z_j\mathfrak m_j(z)}.
\]
\end{remark}

We now formulate the recursion as an equation for the probability law of
a random analytic function. Recall the standard Borel space
\(\mathfrak H\) in \eqref{eq:herglotz-class}. For a probability law \(Q\) on \(\mathfrak H\),
define \(\mathcal F_{d,\alpha}(Q)\) to be the law of
\[
 \Phi_X(h_1,\ldots,h_{d-1})(z)
 =
 -\frac{X_0}
 {z+\displaystyle\sum_{j=1}^{d-1}X_jh_j(z)},
\]
where \(X=(X_0,\ldots,X_{d-1})\sim\operatorname{Dirichlet}(\alpha,\ldots,\alpha)\),
and \(h_1,\ldots,h_{d-1}\) are independent with law \(Q\), independently
of \(X\).  The map
\[
 (X,h_1,\ldots,h_{d-1})
 \longmapsto
 \Phi_X(h_1,\ldots,h_{d-1})
\]
is Borel measurable in the compact-open topology.  The denominator has
positive imaginary part, so
\(\Phi_X(h_1,\ldots,h_{d-1})\in\mathfrak H\).

\begin{theorem}[Unique analytic cavity fixed point]\label{thm:cavity-fixed-point}
The law \(Q_{d,\alpha}\) of the canonical weighted message
\(\mathfrak m_{x\to p}\) is a fixed point of \(\mathcal F_{d,\alpha}\).
It is the unique probability law on \(\mathfrak H\) satisfying
\[
 Q=\mathcal F_{d,\alpha}(Q).
\]
\end{theorem}

\begin{proof}
Existence follows from Proposition~\ref{prop:cavity-rde}. For uniqueness, set
\[
 \kappa_{d,\alpha}
 =\E[X_0(1-X_0)]
 =\frac{\alpha(d-1)}{d(d\alpha+1)}.
\]
Let \(Q\) and \(Q'\) be two fixed points.  Fix
\(z_1,\ldots,z_r\in\C_+\) with
\[
 \eta_\ell:=\Imag z_\ell>\sqrt{\kappa_{d,\alpha}}
 \qquad(1\leq\ell\leq r).
\]
Let \(\nu\) and \(\nu'\) be the joint evaluation laws of \(Q\) and
\(Q'\) at \(z_1,\ldots,z_r\).  On \(\C^r\), use the transportation cost
\[
 c(u,v)=\sum_{\ell=1}^r|u_\ell-v_\ell|
\]
and let \(W_1\) be the associated Wasserstein distance. The evaluation laws
have bounded support, so \(W_1\) is finite and an optimal coupling exists.
Choose an optimal coupling of \(\nu\) and \(\nu'\), and take \(d-1\)
independent copies. Use the same Dirichlet vector in both recursions,
independently of all message couplings. If \(D_\ell\) and \(D'_\ell\) are the
two denominators, then
\[
 |D_\ell|,|D'_\ell|\geq\eta_\ell
\]
and therefore
\[
 \left|
 -\frac{X_0}{D_\ell}
 +\frac{X_0}{D'_\ell}
 \right|
 \leq
 \frac{X_0}{\eta_\ell^2}
 \sum_{j=1}^{d-1}X_j
 |h_j(z_\ell)-h'_j(z_\ell)|.
\]
Taking expectations and using independence gives
\[
 \E\left|
 -\frac{X_0}{D_\ell}
 +\frac{X_0}{D'_\ell}
 \right|
 \leq
 \frac{\kappa_{d,\alpha}}{\eta_\ell^2}
 \E|h(z_\ell)-h'(z_\ell)|.
\]
Consequently, if \(\mathcal F_{\boldsymbol z}\) denotes the induced
map on joint evaluation laws, then
\[
 W_1\!\left(
 \mathcal F_{\boldsymbol z}(\nu),
 \mathcal F_{\boldsymbol z}(\nu')
 \right)
 \leq
 q\,W_1(\nu,\nu'),
 \qquad
 q=\max_{1\leq\ell\leq r}
 \frac{\kappa_{d,\alpha}}{\eta_\ell^2}<1.
\]
Since \(Q\) and \(Q'\) are fixed points, \(\nu\) and \(\nu'\) are fixed
points of \(\mathcal F_{\boldsymbol z}\).  Hence
\(W_1(\nu,\nu')\leq qW_1(\nu,\nu')\), and therefore \(\nu=\nu'\).

Thus all finite-dimensional evaluation laws of \(Q\) and \(Q'\) agree on
the nonempty open half-plane
\[
 \mathcal S=\{z\in\C_+:\Imag z>\sqrt{\kappa_{d,\alpha}}\}.
\]
Choose a countable dense set \(D_0\subset\mathcal S\).  Since \(D_0\) has
accumulation points in \(\C_+\),
Lemma~\ref{lem:analytic-law-identification} gives \(Q=Q'\).
\end{proof}

\begin{remark}[A uniform contraction half-plane]
For \(d\geq3\),
\[
 \kappa_{d,\alpha}<\frac{d-1}{d^2}\leq\frac29<\frac14.
\]
Thus the uniqueness argument works uniformly in \(d\) and \(\alpha\) on the
fixed half-plane \(\Imag z>1/2\).
\end{remark}

The fixed-point law can also be obtained from finite-depth branches.
Starting at the leaves gives the following sequence of laws, which
converges to the infinite-branch message law.

\begin{proposition}[Finite-depth construction]
Let \(Q_0\) be the law of the random analytic function
\[
 z\longmapsto-\frac{Y}{z},
 \qquad
 Y\sim\operatorname{Beta}(\alpha,(d-1)\alpha),
\]
and define recursively
\[
 Q_{R+1}=\mathcal F_{d,\alpha}(Q_R).
\]
Then \(Q_R\) is the law of the weighted cavity message
\(\mathfrak m_R\) on a depth-\(R\) truncated branch obtained by principal
compression, and
\[
 Q_R\Longrightarrow Q_{d,\alpha}
\]
in the topology of locally uniform convergence on \(\C_+\).
\end{proposition}

\begin{proof}
The same block Schur complement gives the finite-depth recursion, with leaf
resolvent \(-1/z\). Couple the truncations with the infinite marked tree. Let
\(P_R\) project onto the branch truncated at depth \(R\), and set
\[
 T_R=P_RT_{x\to p}P_R,
\]
extended by zero to the full branch space. On each finitely supported vector,
\(T_R\) and \(T_{x\to p}\) agree for all sufficiently large \(R\). Since
\(\|T_R\|\leq1\), it follows that \(T_R\to T_{x\to p}\) strongly, and hence
in the strong-resolvent sense.

The cavity messages therefore converge pointwise for every
\(z\in\mathbb C_+\). On each compact subset of \(\mathbb C_+\), the bound
\(|\mathfrak m_R(z)|\leq(\operatorname{Im}z)^{-1}\) is uniform in \(R\).
Vitali's theorem upgrades this pointwise convergence to locally uniform
convergence. Under the coupling, convergence is almost sure and therefore
also holds in law.
\end{proof}

\section*{Statements and Declarations}

\paragraph{Funding.}
This work was supported by JSPS KAKENHI Grant Number 23K12507.

\paragraph{Competing interests.}
The authors have no relevant financial or non-financial interests to disclose.

\paragraph{Author contributions.}
Both authors contributed to the conception and development of the study,
the mathematical analysis, and the preparation of the manuscript.

\paragraph{Data availability.}
No datasets were generated or analyzed during the current study.

\paragraph{Code availability.}
No custom code is required to obtain the results reported in this article.

\begingroup
\small
\sloppy
\hbadness=10000
\emergencystretch=2em
\bibliographystyle{abbrvurl}
\bibliography{references}
\endgroup

\end{document}